\documentclass[11pt,twoside,A4paper]{article} 
\usepackage{amssymb,amsthm,amsmath,amscd,amsfonts,txfonts}
\usepackage{xcolor}

\usepackage[tmargin=1.5cm,bmargin=1.5cm,lmargin=1.7cm,rmargin=1.7cm]{geometry}

\theoremstyle{plain}
\newtheorem{lemma}{Lemma}[section]
\newtheorem{proposition}{Proposition}[section]
\newtheorem{theorem}{Theorem}[section]
\newtheorem{corollary}{Corollary}[section]

\theoremstyle{definition}
\newtheorem{definition}{Definition}[section]

\newtheorem{example}{Example}[section]

\theoremstyle{remark}
\newtheorem{remark}{Remark}[section]

\title{\bfseries\scshape{Classification of some operators on compatible ternary Leibniz algebras}}
\author{
\bfseries\scshape Kol B\'eatrice GAMOU\thanks{E-mail address: \tt{kolbeatrice18@gmail.com}} \\
D\'epartement de Math\'ematiques,\\
Universit\'e Gamal Abdel Nasser, Conakry, Guinea. \\
\bfseries\scshape Ahmed Zahari Abdou  \thanks{e-mail address: zaharymaths@gmail.com}\\
Universit\'{e} de Haute Alsace,\\
 IRIMAS-D\'{e}partement de Math\'{e}matiques,\\
 6, rue des Fr\`eres Lumi\`ere F-68093 Mulhouse, France.\\
\bfseries\scshape  Ibrahima BAKAYOKO \thanks{E-mail address: \tt{ibrahimabakayoko27@gmail.com}}\\
D\'epartement de Math\'ematiques,\\ 
Universit\'e de N'Z\'er\'ekor\'e, Guinea.}

\date{} 

\begin{document} 
\maketitle

\tableofcontents

\begin{abstract} 
In this paper, we establish some basic properties of certain operators (element of centroids, averaging operators, derivations, Nijenhuis operators, 
Rota-Baxter operators) on (compatible) ternary Leibniz algebras and give the  classification of ternary Leibniz algebras, classification of compatible ternary Leibniz algebras. 
Then, we give the descriptions of operators  on found elements of these  classifications.
\end{abstract} 


\section{Introduction}
A ternary algebra consists of a linear vector space $V$ together with a trilinear map $\mu : V\times V\times V\rightarrow V$.
 In other words, ternary algebras are vector spaces equipped with a multiplication with three  items instead of two, as in classical algebraic 
structures. They originate from the works of Jacobson in 1949 in the study of associative algebra $(A, \cdot)$ that
 are closed relative to the ternary operation $[[a, b], c]$, where $[a, b]=ab-ba$. According to some conditions satisfyied by the multiplication 
$\mu$, we dispose of  (partial or total) ternary associative algebras, ternary Leibniz algebras, $3$-Lie algebras, ternary Leibniz-Poisson 
algebras, ternary Hopf algebras, ternary Heap algebras,  Comstrans algebras, Akivis algebras, 
Lie-Yamaguti algebras, Lie triple systems, ternary Jordan algebras, Jordan-Lie triple systems,
 Jordan triple and so on.

It is well-known that mathematical objects are often understood through studying operators defined on them. For instance, in Gallois theory a field
 is studied by its automorphisms, in analysis functions are studied through their derivations,  and in geometry manifolds are studied through 
their vector fields. Fifty years ago, several operators have been found from studies in analysis, probability and combinatorics. Among these
 operators, one can cite, element of centroid, averaging operator, Reynolds operator, Leroux's TD operator, Nijenhuis operator and
 Rota-Baxter operator.

The Rota-Baxter operator originated from the work of G. Baxter on Spitzer's identity in fluctuation theory.
For example, on the polynomial algebra, the indefinite integral
$R(f)(x)=\int_0^xf(t)dt$
 and the inverse of any bijective derivation are Rota-Baxter operators.
Associative Rota- Baxter algebras (associative algebra with Rota-Baxter operator) are used in many fiels of mathematics and mathematical Physics. 
In mathematics, they are used in algebra,
 number theory, operads and combinatorics. In 
mathematical physics they appear as the operator form of the classical Yang Baxter equation or as the fondamental algebraic strucutre in
 the normalisation of quantum field theory of Connes and Kreimer. In non-associative algebra, the Rota-Baxter operators are used  in order 
to produce another one of the same type or not from the previous one.

The Nijenhuis operator on an associative algebra was introduced in \cite{CJ} to study quantum bi-Hamiltonian
systems while the notion Nijenhuis operator on a Lie algebra originated from the concept of Nijenhuis tensor that was introduced by Nijenhuis
in the study of pseudo-complex manifolds and was related to the well known concepts of Schouten-Nijenhuis bracket , the Frolicher-Nijenhuis
 bracket, and the Nijenhuis-Richardson bracket. The associative analog of the Nijenhuis relation may be regaded as the homogeneous version of 
Rota-Baxter relation.

We know that  the classification, up to isomorphism, of any class of algebras is a fundamental and very difficult problem, and
one of the first problem that we encounter when trying to understand the structure of a member of this class of algebras.
Classification of many small-dimensional algebras in the categories of Lie, Leibniz, Jordan, Zinbiel and many
other algebraic structures are done. Among which one can cite  algebraic classifications of 2-dimensional algebras \cite{P},
3-dimensional evolution algebras \cite{CSV}, 3-dimensional anticommutative algebras \cite{KRS, KST}, 3-dimensional diassociative
algebras \cite{RRB}, classification of non-isomorphic complex 3-dimensional transposed Poisson algebras.
\cite{BOK}, 5-dimensional nilpotent, restricted Lie algebras, 6-dimensional nilpotent Lie algebras \cite{CD}, 
6-dimensional solvable
 Lie algebras \cite{T}, 4-dimensional solvable Leibniz algebras \cite{CK}, 5-dimensional solvable Leibniz algebras with three dimensional 
nilradicals \cite{KSTT},  complex Lie algebras up to dimension six  \cite{SW}.

The objective of this work is to give classification of ternary Leibniz algebrasup to dimension four , classification of compatible ternary
 Leibniz algebras. Then to descript operators (element of centroids, averaging operators, derivations, Nijenhuis operators, 
Rota-Baxter operators) on found elements of  classifications.

The paper is structured as follows. In section 2, we give definition and basic properties of (compatible) ternary Leibniz algebras in relation 
with operators (element of centroids, averaging operators, derivations, Nijenhuis operators, Rota-Baxter operators). In section 3 is devoted to the 
classification of ternary Leibniz algebras of dimension less or equal to three. Then we give the classification of  compatible ternary Leibniz 
algebras of dimension up to three. In section 4, we classify derivations, centroids, Rota-Baxter operators, Nijenhuis
operators, averaging operators, Reynolds operators for all algebras in these
dimensions.

The computations for our classifications were done using Mathematica, and we work over the complex field.
\section{Preliminaries and some results}
In this section, we study the invariance of some operators on the (compatible) ternary Leibniz algebra induced by (compatible)
 Leibniz algebra, makes the relationship among these operators and  establish the ternary Leibniz algebra strucutre induced by either
 averaging operators or $O$-operator.
\subsection{Basic notions on Leibniz and ternary Leibniz algebras}

\begin{definition}\label{}
Let $L$ be a vector space over a field $\mathbb{K}$ and  $x, y, z\in  L$.\\
 a) A Leibniz algebra structure on $L$ is a  bilinear map 
$[-, -] : L\otimes L\rightarrow L$   satisfying
 \begin{eqnarray}
   [[x, y], z]=[x, [y, z]]+[[x, z], y] \label{cpa}
 \end{eqnarray}
b)  Two Leibniz algebras $(L, [-, -]_1)$ and $(L, [-, -]_2)$   are called compatible if for any $\lambda_1, \lambda_2\in\mathbb{K}$,  
the following bracket 
\begin{eqnarray}
[x, y]=\lambda_1[x, y]_1+\lambda_2[x, y]_2,  \: \: \forall x, y\in L,\label{cla1}
\end{eqnarray}
  defines a Leibniz algebra structure on $L$.

\end{definition}
See \cite{MS} for examples.

\begin{remark}
 The braket $(\ref{cla1})$ defines a Leibniz algebra structure on $L$ if and only if
 \begin{eqnarray}
[[x, y]_1, z]_2+ [[x, y]_2, z]_1 =[x, [y, z]_1]_2+[x, [y, z]_2]_1 +[[x, z]_1, y]_2+[[x, z]_2, y]_1. \label{cla2}
\end{eqnarray}
\end{remark}

\begin{definition}\label{}
The triple $(L, [-, -]_1, [-, -]_2)$ is said to be a compatible Leibniz algebra if $(L, [-, -]_1)$
and $(L, [-, -]_2)$ are both Leibniz algebras and (\ref{cla2}) holds.
\end{definition}

\begin{definition}
 Let $(L, [-, -])$ be a Leibniz algebra. A linear map $\beta : L\rightarrow L$ is said to be an averaging operator if 
 $$\beta([\beta(x), y])=[\beta(x), \beta(y)]=\beta([x, \beta(y)]),$$
for any $x, y\in L$.
\end{definition}

\begin{definition}
 Let $(L, [-, -])$ be a Leibniz algebra. A linear map $N : L\rightarrow L$ is said to be 
a Nijenhuis operator if 
$$[N(x), N(y)]=N\Big([N(x), y]+[x, N(y)]-N([x, y])\Big),$$
for any $x, y\in L$.
\end{definition}
The following result comes from a direct computation.
\begin{proposition}\label{}
Let $(L, [-, -])$ be a Leibniz algebra. Then, $(L, [-, -], [-, -]_N)$ is a compatible Leibniz algebra, where for any $x, y\in L$,
$$
[x, y]_N=[N(x), y]+[x, N(y)]-[x, y].
$$
\end{proposition}

\begin{definition}\label{dt}
 A ternary algebra $L$ is said to be a ternary Leibniz  algebra if the bracket satisfies
 the following identity :
\begin{eqnarray}
 \llbracket\llbracket x, y, z\rrbracket, t, u\rrbracket=\llbracket x, y, \llbracket z, t, u\rrbracket\rrbracket+\llbracket x,
 \llbracket y, t, u\rrbracket, z\rrbracket+\llbracket\llbracket x, t, u\rrbracket, y, z\rrbracket,\label{lci}
\end{eqnarray}
for any $x, y, z, t, u\in L$.
\end{definition}

The following result asserts that one may associate a ternary Leibniz algebra to a Leibniz algebra. It will be very useful later for this work.
\begin{theorem}(\cite{IS})\label{ll3}
Let $(L, [-, -])$ be a Leibniz algebra. Then 
$$T(L)=(L, \{x, y, z\}:=[x, [y, z]]),$$
is a ternary Leibniz algebra, for  any $x, y, z\in L$.
\end{theorem}

\begin{definition}\label{dc}
 A compatible ternary Leibniz algebra is a triple $(L, \llbracket-, -, -\rrbracket_1, \llbracket-, -, -\rrbracket_2)$ where 
 $(L, \llbracket-, -, -\rrbracket_1)$
and $(L, \llbracket-, -, -\rrbracket_2)$ are two ternary Leibniz algebras such that :
\begin{eqnarray}
 \llbracket\llbracket x, y, z\rrbracket_1, t, u\rrbracket_2+\llbracket\llbracket x, y, z\rrbracket_2, t, u\rrbracket_1
&=&\llbracket x, y, \llbracket z, t, u\rrbracket_1\rrbracket_2+\llbracket x, y, \llbracket z, t, u\rrbracket_2\rrbracket_1
+\llbracket x, \llbracket y, t, u\rrbracket _1, z\rrbracket_2\nonumber\\
&&+\llbracket x, \llbracket y, t, u\rrbracket _2, z\rrbracket_1+\llbracket\llbracket x, t, u\rrbracket_1, y, z\rrbracket_2
+\llbracket\llbracket x, t, u\rrbracket_2, y, z\rrbracket_1, \label{cc}
\end{eqnarray}
for any $x, y, z\in L$.
\end{definition}

The proof of the next proposition is straightforward.
\begin{proposition}
 The triple $(L, \llbracket-, -, -\rrbracket_1, \llbracket-, -, -\rrbracket_2)$ is a compatible ternary Leibniz algebra if and only
 if the new bilinear map 
$$\llbracket x, y, z\rrbracket=k_1\llbracket-, -, -\rrbracket_1+k_2\llbracket-, -, -\rrbracket_2$$
define a ternary Leibniz structure on $L$.
\end{proposition}

\subsection{Generalized derivations on compatible ternary Leibniz algebras}
\begin{definition}
\begin{enumerate}
 \item [1)]
Let $(L, [-, -])$ be a Leibniz algebra. A
linear map $D\in End(L)$ is said to be a central derivation if 
$$D([x, y])=[D(x), y]=[x, D(y)]=0,$$
for all $x, y\in L$.
\item [2)] A central derivation on a compatible Leibniz algebra $(L, [-, -]_1, [-, -]_2)$ is a central derivation on both 
$(L, [-, -]_1)$ and $(L, [-, -]_2)$.
\item [3)]
Let $(L, \llbracket-, -, -\rrbracket)$ be a ternary Leibniz algebra. A
linear map $D\in End(L)$ is said to be a central derivation if
$$D(\llbracket x, y, z\rrbracket) =\llbracket D(x), y, z\rrbracket=\llbracket x, D(y), z\rrbracket=\llbracket x, y, D(z)\rrbracket=0,$$
for all $x, y, z\in L$.
\item [4)]
 A central derivation on a compatible Leibniz algebra $(L, [-, -, -]_1, [-, -, -]_2)$ is a central derivation on both 
$(L, [-, -, -]_1)$ and $(L, [-, -, -]_2)$.
\end{enumerate}
\end{definition}

\begin{lemma}
 Let $D : L\rightarrow L$ be a central derivation of a Leibniz algebra $(L, [-, -])$.
Then $D$ is also a central derivation of the associated ternary Leibniz algebra $T(L)$. 
\end{lemma}
\begin{proof}
 For any $x, y, z\in L$, 
\begin{eqnarray}
D(\{x, y, z\})&=&D([x, [y, z]])=[D(x), [y, z]]=[x, D([y, z])]=[x, [D(y), z]]=[x, [y, D(z)]]=0\nonumber\\
&=&\{D(x), y, z\}=\{x, D(y), z\}=\{x, y, D(z)\}=0.\nonumber
\end{eqnarray}
Which implies $D$ is a central derivation of the ternary Leibniz algebra $T(L)$.
\end{proof}

\begin{proposition}
  Let $(L, [-, -]_1)$ and $(L, [-, -]_2)$ be two compatible Leibniz algebras such that $(T(L), \llbracket-, -, -\rrbracket_1)$ and 
$(T(L), \llbracket-, -, -\rrbracket_1)$ be compatibles.
Let $D : L\rightarrow L$ be a central derivation of the compatible Leibniz algebra $(L, [-, -]_1, [-, -]_2)$ i.e., for any $x, y\in L$,
$$D([x, y]_i)=[D(x), y]_i=[x, D(y)]_i=0, i=1, 2.$$
Then $D$ is also a central derivation on the compatible ternary Leibniz algebra $T(L)$. 
\end{proposition}
\begin{definition}
\begin{enumerate}
 \item [1)]
Let $(L, [-, -])$ be a Leibniz algebra. A
linear map $D\in End(L)$ is said to be a generalized derivation if there exists two linear maps $D', D''\in End(L)$  such that
$$D''([x, y]) =[D(x), y]+[x, D'(y)],$$
for all $x, y\in L$.
\item [2)] 
A generalized derivation on a compatible Leibniz algebra $(L, [-, -]_1, [-, -]_2)$ is a generalized derivation on both $(L, [-, -]_1)$ and $(L, [-, -]_2)$.
\item [3)]
Let $(L, [-, -, -])$ be a ternary Leibniz algebra. A
linear map $D\in End(L)$ is said to be a generalized derivation if there
exists three linear maps $D', D'', D'''\in End(L)$  such that
$$D'''(\llbracket x, y, z\rrbracket ) =\llbracket D(x), y, z\rrbracket+\llbracket x, D'(y), z\rrbracket+\llbracket x, y, D''(z)\rrbracket,$$
for all $x, y, z\in L$.
\item [4)]
A generalized derivation on a compatible Leibniz algebra $(L, [-, -, -]_1, [-, -, -]_2)$ is a generalized derivation on
 both $(L, [-, -, -]_1)$ and $(L, [-, -, -]_2)$.
\end{enumerate}
\end{definition}
\begin{lemma}\label{dl}
 Let $D : L\rightarrow L$ be a generalized derivation on a Leibniz algebra $(A, [-, -])$ such that
 \begin{eqnarray}
  D''([x, y])&=&[D(x), y]+[x, D'(y)],\nonumber\\
D'''([x, y])&=&[D(x), y]+[x, D''(y)],\nonumber
 \end{eqnarray}
for any $x, y\in L$. Then, $D$ is a generalized derivation of the ternary Leibniz algebra $T(L)$.  
\end{lemma}
\begin{proof}
 For any $x, y, z\in L$, 
\begin{eqnarray}
 D'''(\{x, y, z\})&=&D'''([x, [y, z]])=[D(x), [y, z]]+[x, D''([y, z])]\nonumber\\
&=&[D(x), [y, z]]+ [x, [D(y), z]+ [y, D'(z)]]\nonumber\\
&=&[D(x), [y, z]]+ [x, [D(y), z]]+[x, [y, D'(z)]]\nonumber\\
&=&\{D(x), y, z\}+ \{x, D(y), z\}+\{x, y, D'(z)\}\nonumber.
\end{eqnarray}
Therefore, $D$ is a generalized derivation of $T(L)$.
\end{proof}

\begin{proposition}
  Let $(L, [-, -]_1)$ and $(L, [-, -]_2)$ be two compatible Leibniz algebras such that the associated ternary Leibniz algebras be compatible.
 Let $D : L\rightarrow L$ be a generalized derivation on the compatible Leibniz algebra $(L, [-, -]_1, [-, -]_2)$ such that
 \begin{eqnarray}
  D''([x, y]_i)&=&[D(x), y]_i+[x, D'(y)]_i,\nonumber\\
D'''([x, y]_i)&=&[D(x), y]_i+[x, D''(y)]_i,\nonumber
 \end{eqnarray}
for any $x, y\in L$. Then, $D$ is a generalized derivation of the compatible ternary Leibniz algebra $T(L)$.  
\end{proposition}
\begin{proof}
 It follows from Lemma \ref{dl}.
\end{proof}

\begin{definition}
\begin{enumerate}
 \item [1)]
Let $(L, [-, -])$ be a Leibniz algebra. A
linear map $D\in End(L)$ is said to be a quasiderivation if there exists a linear map $D'\in End(L)$  such that
$$D'([x, y]) =[D(x), y]+[x, D(y)],$$
for all $x, y\in L$.
\item [2)] A quasiderivation on a compatible Leibniz algebra $(L, [-, -]_1, [-, -]_2)$ is a quasiderivation on both $(L, [-, -]_1)$ 
and $(L, [-, -]_2)$.
\item [3)]
Let $(L, [-, -, -])$ be a ternary Leibniz algebra. A
linear map $D\in End(L)$ is said to be a quasiderivation if there
exists a linear map $D'\in End(L)$  such that
$$D'(\llbracket x, y, z\rrbracket) =\llbracket D(x), y, z\rrbracket+\llbracket x, D(y), z\rrbracket+\llbracket x, y, D(z)\rrbracket,$$
for all $x, y\in L$.
\item [4)] A quasiderivation on a compatible ternary Leibniz algebra $(L, [-, -, -]_1, [-, -, -]_2)$ is a quasiderivation on both $(L, [-, -, -]_1)$
 and $(L, [-, -, -]_2)$.
\end{enumerate}
\end{definition}

\begin{lemma}\label{qsd}
 Let $D, D' : L\rightarrow L$ be two quasi-derivations on a Leibniz algebra $(L, [-, -])$ such that
 \begin{eqnarray}
  D'([x, y])&=&[D(x), y]+[x, D(y)],\nonumber\\
D''([x, y])&=&[D'(x), y]+[x, D'(y)],\nonumber
 \end{eqnarray}
for any $x, y\in L$. Then, $D$ is a quasi-derivation of the ternary Leibniz algebra $T(L)$.  
\end{lemma}
\begin{proof}
 For all $x, y\in L$, 
 \begin{eqnarray}
D''(\{x, y, z\})&=&D''([[x, y], z])=[D'[x, y], z]+[[x, y], D'(z)]\nonumber\\
&=&[[D(x), y]+ [x, D(y)], z]+[[x, y], D'(z)]\nonumber\\
&=&\{D(x), y, z\}+ \{x, D(y), z\}+\{x, y, D'(z)\}\nonumber.
\end{eqnarray}
This proves the assertion.
\end{proof}

\begin{proposition}
  Let $(L, [-, -]_1)$ and $(L, [-, -]_2)$ be two compatible Leibniz algebras such that the associated ternary Leibniz algebras be compatible.
 Let $D, D' : L\rightarrow L$ be two quasi-derivations on the compatible Leibniz algebra $(L, [-, -]_1, [-, -]_2)$ such that
 \begin{eqnarray}
  D'([x, y]_i)&=&[D(x), y]_i+[x, D(y)]_i,\nonumber\\
D''([x, y]_i)&=&[D'(x), y]_i+[x, D'(y)]_i,\nonumber
 \end{eqnarray}
for any $x, y\in L$. Then, $D$ is a quasi-derivation of the compatible ternary Leibniz algebra $T(L)$.  
\end{proposition}
\begin{proof}
 The proof follows from Lemma \ref{qsd} by direct calculation.
\end{proof}

\begin{definition}\label{dlt}
1) A linear map $D: L\rightarrow L$ on a  Leibniz algebra is said to be a {\bf derivation} of weight $\lambda\in\mathbb{K}$ if
$$D([x, y])=[D(x), y]+[x, D(y)]+\lambda[x, y],$$
for any $x, y\in L$.\\
2)  A derivation on a compatible Leibniz algebra $(L, [-, -]_1, [-, -]_2)$ is a derivation on both $(L, [-, -]_1)$ and $(L, [-, -]_2)$.\\
3)  A linear map $D: L\rightarrow L$ on a ternary Leibniz algebra is called
 a {\bf derivation of weight $\lambda\in\mathbb{K}$} if 
\begin{eqnarray}
 &&D(\llbracket x, y, z\rrbracket)=\llbracket D(x), y, z\rrbracket+\llbracket x, D(y), z\rrbracket+\llbracket x, y, D(z)\rrbracket\nonumber\\
&&\qquad\qquad\qquad\qquad+\lambda\llbracket D(x), D(y), z\rrbracket+\lambda\llbracket x, D(y), D(z)\rrbracket
+\lambda\llbracket D(x), y, D(z)\rrbracket+\lambda^2\llbracket D(x), D(y), D(z)\rrbracket,\label{dev}
\end{eqnarray}
for any $x, y, z\in L$.\\
4) A derivation on a compatible ternary Leibniz algebra $(L, [-, -, -]_1, [-, -, -]_2)$ is a derivation on both 
$(L, [-, -, -]_1)$ and $(L, [-, -, -]_2)$.
\end{definition}

The following proposition comes from direct computation.
\begin{proposition}
 Let $D : L\rightarrow L$ be a derivation on a Leibniz algebra $(L, [-, -])$.
Then, $D$ is a generalized derivation of the ternary Leibniz algebra $T(L)$.\\
Moreover, if $D$ is also a derivation on another Leibniz algebra $(L, [-, -]_1)$ which is compatible with $(L, [-, -])$
 such that the associated ternary Leibniz algebras are compatibles,
then, $D$ is a derivation of the compatible ternary Leibniz algebra $T(L)$.
\end{proposition}
\subsection{Averaging operators on compatible ternary Leibniz algebras}
\begin{definition}\label{av}
 1) A  linear map $\beta : L\rightarrow L$ is said to be an averaging operator on a ternary Leibniz algebra $(L, \llbracket -, -, -)$ if, for any $x, y, z\in L$,
\begin{eqnarray}
 \beta(\llbracket\beta(x), \beta(y), z\rrbracket)=\beta(\llbracket\beta(x), y, \beta(z)\rrbracket)
=\beta(\llbracket x, \beta(y), \beta(z)\rrbracket)=\llbracket\beta(x), \beta(y), \beta(z)\rrbracket.\label{av1}
\end{eqnarray}
2) An averaging operator on a compatible ternary Leibniz algebra $(L, [-, -, -]_1, [-, -, -]_2)$ is an averaging
 operator  on both 
$(L, [-, -, -]_1)$ and $(L, [-, -, -]_2)$.
\end{definition}
\begin{lemma}\label{avt}
 Let $(L, \llbracket-, -, -\rrbracket, \beta)$ be an injective averaging ternary Leibniz algebra. Then, $L$, endowed with the new multiplication 
$\llbracket-, -, -\rrbracket_\beta : L\times L\times L\rightarrow L$ defined by 
\begin{eqnarray}
 \llbracket x, y, z\rrbracket_\beta:=\llbracket \beta(x), \beta(y), z\rrbracket,
\end{eqnarray}
for all $x, y, z\in L$, makes $L$ into a ternary Leibniz algebra.
\end{lemma}
\begin{proof}
 For any $x, y, z\in L$,
\begin{eqnarray}
&&\qquad\beta\Big(\llbracket\llbracket x, y, z\rrbracket_\beta, t, u\rrbracket_\beta-\llbracket x, y, \llbracket z, t, u\rrbracket_\beta\rrbracket_\beta
-\llbracket x, \llbracket y, t, u\rrbracket_\beta, z\rrbracket_\beta-\llbracket\llbracket x, t, u\rrbracket_\beta, y, z\rrbracket_\beta\Big)\nonumber\\
&&=\beta\Big(\llbracket\beta(\llbracket\beta(x), \beta(y), z\rrbracket), \beta(t), u\rrbracket-\llbracket\beta(x), \beta(y), \llbracket\beta(z), \beta(t), u\rrbracket\rrbracket-
\llbracket\beta(x), \beta(\llbracket y, t, u\rrbracket), z\rrbracket-\llbracket\beta(\llbracket x, t, u\rrbracket), \beta(y), z\rrbracket\Big)\nonumber\\
&&=\llbracket\beta(\llbracket\beta(x), \beta(y), z\rrbracket), \beta(t), \beta(u)\rrbracket-\llbracket\beta(x), \beta(y), \beta(\llbracket\beta(z), \beta(t), u\rrbracket)\rrbracket\nonumber\\
&&\qquad-\llbracket\beta(x), \beta(\llbracket\beta(y), \beta(t), u\rrbracket), \beta(z)\rrbracket-\llbracket\beta(\llbracket\beta(x), \beta(t), u\rrbracket), \beta(y), \beta(z)\rrbracket\nonumber\\
&&=\llbracket\llbracket\beta(x), \beta(y), \beta(z)\rrbracket, \beta(t), \beta(u)\rrbracket-\llbracket\beta(x), \beta(y), \llbracket\beta(z), \beta(t), \beta(u)\rrbracket\rrbracket\nonumber\\
&&\qquad-\llbracket\beta(x), \llbracket\beta(y), \beta(t), \beta(u)\rrbracket), \beta(z)\rrbracket-\llbracket\llbracket\beta(x), \beta(t), \beta(u)\rrbracket), \beta(y), \beta(z)\rrbracket\nonumber.
\end{eqnarray}
The right hand side vanishes by ternary Leibniz rule, and the conclusion follows from injectivity.
\end{proof}

\begin{proposition}
Let $(L, \llbracket-, -, -\rrbracket_1, \beta_1)$ and $(L, \llbracket-, -, -\rrbracket_2, \beta_2)$ be two commuting
injective averaging ternary Leibniz algebras  such that
$$\llbracket\beta_1(x), \beta_1(y), \beta_1(z)\rrbracket_1=\llbracket\beta_2(x), \beta_2(y), \beta_2(z)\rrbracket_2,$$
for all $x, y, z\in L$. Then, with notation of Lemma \ref{avt}, $(L, \llbracket-, -, -\rrbracket_{\beta_1}, \llbracket-, -, -\rrbracket_{\beta_2})$ is a compatible ternary Leibniz algebra.
\end{proposition}
\begin{proof}
 It follows from direct computation of 
\begin{eqnarray}
 &&\beta_1\beta_2\Big(\llbracket\llbracket x, y, z\rrbracket_{\beta_1}, t, u\rrbracket_{\beta_2}+\llbracket\llbracket x, y, z\rrbracket_{\beta_2}, t, u\rrbracket_{\beta_1}
-\llbracket x, y, \llbracket\llbracket z, t, u\rrbracket_{\beta_1}\rrbracket_{\beta_2}-\llbracket x, y, \llbracket z, t, u\rrbracket_2\rrbracket_{\beta_1}-\llbracket x, \llbracket y, t, u\rrbracket_{\beta_1}, z\rrbracket_{\beta_2}\nonumber\\
&&\hspace{5cm}-\llbracket x, \llbracket y, t, u\rrbracket_{\beta_2}, z\rrbracket_{\beta_1}-\llbracket\llbracket x, t, u\rrbracket_{\beta_1}, y, z\rrbracket_{\beta_2}-\llbracket\llbracket x, t, u\rrbracket_{\beta_2}, y, z\rrbracket_{\beta_1}\Big),\nonumber
\end{eqnarray}
by using the assumption.
\end{proof}

Now, we introduce compatible averaging operator. It is inspired from (\cite{AD}).
\begin{definition}
Two averaging operators on a Leibniz algebras $L$ are said to be compatible if their sum is also an averaging operator on $L$.
This means that 
 \begin{eqnarray}
  \llbracket\beta_2(\llbracket\beta_1(x), y\rrbracket), z\rrbracket+\llbracket\beta_1(\llbracket\beta_2(x), y\rrbracket), z\rrbracket=\llbracket\llbracket\beta_1(x), \beta_2(y)\rrbracket, z\rrbracket+\llbracket\llbracket\beta_2(x), \beta_1(y)\rrbracket, z\rrbracket \label{cao}
 \end{eqnarray}
for all $x, y, z\in L$.
\end{definition}

We have already proved that $[-, -]_\beta$ is a Leibniz algebra (\cite{IS}, Proposition 3). Now we have the below proposition.
\begin{proposition}\label{}
Let $(L, [-, -])$ be a Leibniz algebra and $\beta_1, \beta_2 :L\rightarrow L$ two compatible commuting injective averaging 
operators on $L$. Then, $(L, [-, -]_{\beta_1}, [-, -]_{\beta_2})$ is a compatible Leibniz algebra.
\end{proposition}
\begin{proof}
The compatibility comes from
 a direct computation.
\end{proof}

\begin{proposition}
 Let $(L, \llbracket-, -, -\rrbracket_1, \llbracket-, -, -\rrbracket_2)$ be a compatible ternary Leibniz algebra and $\beta : L\rightarrow L$ be an injective averaging operator
on $L$ for both $\llbracket-, -, -\rrbracket_1$ and $\llbracket-, -, -\rrbracket_2)$. Then, $$L^\beta :=(L, \llbracket-, -, -\rrbracket_1^\beta, \llbracket-, -, -\rrbracket_2^\beta)$$ is a compatible ternary Leibniz
algebra, where $\llbracket-, -, -\rrbracket_i^\beta=[\beta(-), \beta(-), -]_i, i=1,2$. Moreover, $\beta : (L, \llbracket-, -, -\rrbracket^\beta)\rightarrow (L, \llbracket-, -, -\rrbracket)$ is a
morphism of ternary Leibniz algebra, with $\llbracket-, -, -\rrbracket=\llbracket-, -, -\rrbracket_1+\llbracket-, -, -\rrbracket_2$.
\end{proposition}
\begin{proof}
It is clear, from Lemma \ref{avt}, that $L$ is a ternary Leibniz algebra for each of the bracket $[-, -, -]_1^\beta$ and $[-, -, -]_2^\beta$.
Then, for any $x, y, z\in L$,
\begin{eqnarray}
 &&\qquad\beta\Big(\llbracket\llbracket x, y, z\rrbracket_1^\beta, t, u\rrbracket_2^\beta+\llbracket\llbracket x, y, z\rrbracket_2^\beta, t, u\rrbracket_1^\beta
-\llbracket x, y, \llbracket z, t, u\rrbracket_1^\beta\rrbracket_2^\beta- \llbracket x, y, \llbracket z, t, u\llbracket _2^\beta\rrbracket_1^\beta-\llbracket x, \llbracket y, t, u\rrbracket_1^\beta, z\rrbracket_2^\beta\nonumber\\
&&\qquad-\llbracket x, \llbracket y, t, u\rrbracket_2^\beta, z\rrbracket_1^\beta-\llbracket\llbracket x, t, u\rrbracket_1^\beta, y, z\rrbracket_2^\beta-\llbracket\llbracket x, t, u\rrbracket_2\beta, y, z\rrbracket_1^\beta\Big)=\nonumber\\
&&=\beta\Big(\llbracket\beta(\llbracket\beta(x), \beta(y), z\rrbracket_1), \beta(t), u\rrbracket_2+\llbracket\beta(\llbracket\beta(x), \beta(y), z\rrbracket_2), \beta(t), u\rrbracket_1
-\llbracket\beta(x), \beta(y), \llbracket\beta(z), \beta(t), u\rrbracket_1\rrbracket_2\nonumber\\
&&\qquad- \llbracket\beta(x), \beta(y), \llbracket\beta(z), \beta(t), u\rrbracket_2\rrbracket_1
-\llbracket\beta(x), \beta(\llbracket\beta(y), \beta(t), u\rrbracket_1), z\rrbracket_2-\llbracket\beta(x), \beta(\llbracket\beta(y), \beta(t), u\rrbracket_2), z\rrbracket_1\nonumber\\
&&\qquad-\llbracket\beta(\llbracket\beta(x), \beta(t), u\rrbracket_1, \beta(y), z\rrbracket_2-\llbracket\beta(\llbracket\beta(x), \beta(t), u\rrbracket_2, \beta(y), z\rrbracket_1\Big)\nonumber\\
&&=\llbracket\beta(\llbracket\beta(x), \beta(y), z\rrbracket_1), \beta(t), \beta(u)\rrbracket_2+\llbracket\beta(\llbracket\beta(x), \beta(y), z\rrbracket_2), \beta(t), \beta(u)\rrbracket_1
-\llbracket\beta(x), \beta(y), \beta(\llbracket\beta(z), \beta(t), u\rrbracket_1)\rrbracket_2\nonumber\\
&&\qquad- \llbracket\beta(x), \beta(y), \beta(\llbracket\beta(z), \beta(t), u\rrbracket_2)\rrbracket_1
-\llbracket\beta(x), \beta(\llbracket\beta(y), \beta(t), u\rrbracket_1, \beta(z)\rrbracket_2-\llbracket\beta(x), \beta(\llbracket\beta(y), \beta(t), u\rrbracket_2), \beta(z)\rrbracket_1\nonumber\\
&&\qquad-\llbracket\beta(\llbracket\beta(x), \beta(t), u\rrbracket_1), \beta(y), \beta(z)\rrbracket_2-\llbracket\beta(\llbracket\beta(x), \beta(t), u\rrbracket_2, \beta(y), \beta(z)\rrbracket_1\nonumber\\
&&=\llbracket\llbracket\beta(x), \beta(y), \beta(z)\rrbracket_1, \beta(t), \beta(u)\rrbracket_2+\llbracket\llbracket\beta(x), \beta(y), \beta(z)\rrbracket_2), \beta(t), \beta(u)\rrbracket_1
-\llbracket\beta(x), \beta(y), \llbracket\beta(z), \beta(t), \beta(u)\rrbracket_1)\rrbracket_2\nonumber\\
&&\qquad- \llbracket\beta(x), \beta(y), \llbracket\beta(z), \beta(t), \beta(u)\rrbracket_2\rrbracket_1
-\llbracket\beta(x), \llbracket\beta(y), \beta(t), \beta(u)\rrbracket_1, \beta(z)\rrbracket_2-\llbracket\beta(x), \llbracket\beta(y), \beta(t), \beta(u)\rrbracket_2), \beta(z)\rrbracket_1\nonumber\\
&&\qquad-\llbracket\llbracket\beta(x), \beta(t), \beta(u)\rrbracket_1, \beta(y), \beta(z)\rrbracket_2-\llbracket\llbracket\beta(x), \beta(t), \beta(u)\rrbracket_2, \beta(y), \beta(z)\rrbracket_1\nonumber
\end{eqnarray}
The right hand side vanishes by compatibility condition, and the conclusion follows from injectivity.\\
Now,
\begin{eqnarray}
\beta(\llbracket x, y, z\rrbracket^\beta)&=&\beta\Big(\llbracket x, y, z\rrbracket_1^\beta+\llbracket x, y, z\rrbracket_2^\beta\Big)
=\beta\Big(\llbracket\beta(x), \beta(y), z\rrbracket_1+\llbracket\beta(x), \beta(y), z\rrbracket_2\Big)\nonumber\\
&=&\llbracket\beta(x), \beta(y), \beta(z)\rrbracket_1+\llbracket\beta(x), \beta(y), \beta(z)\rrbracket_2\nonumber\\
&=&\llbracket\beta(x), \beta(y), \beta(z)\rrbracket\nonumber.
\end{eqnarray}
This achieves the proof.
\end{proof}

\begin{remark}
 If $\beta$ is an averaging operator on a Leibniz algebra $L$, then it is also an averaging operator on the ternary Leibniz algebra $T(L)$.
\end{remark}

We have the following consequence :
\begin{corollary}
 Let $(L, [-, -]_1, [-, -]_2)$ be a compatible Leibniz algebra such that the associated ternary Leibniz algebra is compatible and 
$\beta : L\rightarrow L$ an injective averaging operator on $L$. Then, 
$(L, [-, [-, -]_1^\beta]_1^\beta, [-, [-, -]_2^\beta]_2^\beta)$ is a compatible ternary Leibniz algebra.
\end{corollary}
\begin{proof}
 It follows from that $[-, -]^\beta=[\beta(-), -]$ is a Leibniz algebra and
\begin{eqnarray}
 \llbracket x, y, z\rrbracket_i^\beta=[ x, [ y, z]_i^\beta]_i^\beta=[\beta(x), [\beta(y), z]_i]_i=\llbracket\beta(x), \beta(y), z\rrbracket_i,\nonumber
\end{eqnarray}
for any $x, y, z\in L$.
\end{proof}

\begin{proposition}
 Let us denote $\mathcal C_c(L)$ the category of all compatible ternary Leibniz structures on the vector space $L$.
 On $\mathcal C_c(L)$, let us define the relation 
$$(L, [-, -]_1)\mathcal R (L, [-, -]_2)\quad \mbox{if and only if}\quad (L, [-, -]_1, [-, -]_2)\quad \mbox{is a ternary Leibniz algebra}$$
The relation $\mathcal R$ is an equivalence relation on $\mathcal C_c(L)$.
\end{proposition}
\begin{proof}
 It is trivial by construction of $\mathcal C_c(L)$.
\end{proof}

\subsection{Rota-Baxter and Nijenhuis operators on compatible ternary Leibniz algebras}

\begin{definition}\label{rbc}
1) A linear map $R: L\rightarrow L$ on a  Leibniz algebra is said to be a {\bf Rota-Baxter operator} of weight $\lambda\in\mathbb{K}$ if
$$[R(x), R(y)])=R\Big([R(x), y]+[x, R(y)]+\lambda[x, y]),$$
for any $x, y\in L$.\\
2)  A  Rota-Baxter operator of weight $\lambda$ on a compatible Leibniz algebra $(L, [-, -]_1, [-, -]_2)$ 
is a Rota-Baxter operator of weight $\lambda$ on both $(L, [-, -]_1)$ and $(L, [-, -]_2)$.\\
3)  A linear map $R: L\rightarrow L$ on a ternary Leibniz algebra is called
 a Rota-Baxter operator of weight $\lambda$ if 
\begin{eqnarray}
 &&\llbracket R(x), R(y), R(z)\rrbracket=R\Big(\llbracket R(x), R(y), z\rrbracket+\llbracket R(x), y, R(z)\rrbracket
+\llbracket x, R(y), R(z)\rrbracket\nonumber\\
&&\qquad\qquad\qquad\qquad+\lambda\llbracket R(x), y, z\rrbracket+\lambda\llbracket x, R(y), z\rrbracket
+\lambda\llbracket x, y, R(z)\rrbracket+\lambda^2
\llbracket x, y, z\rrbracket\Big),\label{rotid}
\end{eqnarray}
for any $x, y, z\in L$.\\
4) A Rota-Baxter operator of weight $\lambda$ on a compatible ternary Leibniz algebra $(L, [-, -, -]_1, [-, -, -]_2)$ is a Rota-Baxter
 operator of weight $\lambda$
 on both 
$(L, [-, -, -]_1)$ and $(L, [-, -, -]_2)$.
\end{definition}

\begin{definition}\label{njc}
1) A linear map $N: L\rightarrow L$ on a  Leibniz algebra is said to be a {\bf Nijenhuis operator}  if
$$[N(x), N(y)])=N\Big([N(x), y]+[x, N(y)]-N([x, y])),$$
for any $x, y\in L$.\\
2)  A  Rota-Baxter operator of weight $\lambda$ on a compatible Leibniz algebra $(L, [-, -]_1, [-, -]_2)$ 
is a Rota-Baxter operator of weight $\lambda$ on both $(L, [-, -]_1)$ and $(L, [-, -]_2)$.\\
3)  A linear map $N: L\rightarrow L$ on a ternary Leibniz algebra is called
 a Nijenhuis operator on $L$ if 
\begin{eqnarray}
 &&\llbracket N(x), N(y), N(z)\rrbracket=N\Big(\llbracket N(x), N(y), z\rrbracket+\rrbracket N(x), y, N(z)\rrbracket
+\llbracket x, N(y), N(z)\rrbracket\Big)\nonumber\\
&&\qquad\qquad\qquad\qquad-N^2\Big(\llbracket N(x), y, z\rrbracket
+\llbracket x, N(y), z\rrbracket+\llbracket x, y, N(z)\rrbracket\Big)+N^3(\llbracket x, y, z\rrbracket)\Big),\label{nij}
\end{eqnarray}
for any $x, y, z\in L$.\\
4) A Rota-Baxter operator of weight $\lambda$ on a compatible ternary Leibniz algebra $(L, [-, -, -]_1, [-, -, -]_2)$ is a Rota-Baxter
 operator of weight $\lambda$
 on both 
$(L, [-, -, -]_1)$ and $(L, [-, -, -]_2)$.
\end{definition}

The following proposition comes from straightforward computation.
\begin{proposition}
 Let $R : L\rightarrow L$ ($N$ respectively) be a Rota-Baxter operator of weight $\lambda$ (Nijenhuis operator respectively)
 on a Leibniz algebra $(L, [-, -])$.
Then, $R$ ($N$ respectively) is a Rota-Baxter operator of weight $\lambda$ (Nijenhuis operator respectively) of the ternary Leibniz algebra
 $T(L)$.\\
Moreover, if $R$ ($N$ respectively) is also a Rota-Baxter operator of weight $\lambda$ (Nijenhuis operator respectively) on another Leibniz 
algebra $(L, [-, -]_1)$ which is compatible with $(L, [-, -])$ such that the associated ternary Leibniz algebras are compatibles,
then, $R$ ($N$ respectively) is a Rota-Baxter operator of weight $\lambda$ (Nijenhuis operator respectively) on the compatible ternary Leibniz 
algebra $T(L)$.
\end{proposition}
Now, we have :
\begin{proposition}
 Let $(L, \llbracket-, -, -\rrbracket)$ be a ternary Leibniz algebra and $R: L\rightarrow L$ be an inversible Rota-Baxter operator of weight
 $\lambda\in\mathbb{K}$ on 
$L$. Then, $R^{-1}: L\rightarrow L$ is a derivation of weight $\lambda$.  
\end{proposition}
\begin{proof}
 It follows from Rota-Baxter identity (\ref{rotid}) by putting $R(x)=x', R(y)=y', R(z)=z'$.
\end{proof}

The following proposition connects Rota-Baxter operators to derivations.
\begin{proposition}
 Let $(L, \llbracket-, -, -\rrbracket)$ be a ternary Leibniz algebra and $N : L\rightarrow L$ be a linear operator.
\begin{enumerate}
 \item [1)] If $N^2=0$, then $N$ is a Nijenhuis operator if and only if $N$ is a Rota-Baxter operator of weight $0$.
\item [2)] If $N^2=N$ then $N$ is a Nijenhuis operator if $N$ is a Rota-Baxter operator of weight $-1$.
\end{enumerate}
\end{proposition}
\begin{proof}
\begin{enumerate}
 \item [1)]
As $N^2=0$ and $N$ is a Nijenhuis operator, we have 
\begin{eqnarray}
 \llbracket N(x), N(y), N(z)\rrbracket &=&N\Big(\llbracket N(x), N(y), z\rrbracket +\llbracket N(x), y, N(z)\rrbracket+\llbracket x, N(y), N(z)\rrbracket\Big)\nonumber\\
 &&\quad-N^2\Big(\llbracket N(x), y, z\rrbracket+\llbracket x, N(y), z\rrbracket+\llbracket x, y, N(z)\rrbracket\Big)+N^3(\llbracket x, y, z\rrbracket),\nonumber\\
&=&N\Big(\llbracket N(x), N(y), z\rrbracket+\llbracket N(x), y, N(z)\rrbracket+\llbracket x, N(y), N(z)\rrbracket\Big).\nonumber
\end{eqnarray}
Which means that $N$ is a Rota-Baxter operator of weight $0$. The converse is proved by going up the egalities.
\item [2)] Since $N^2=N$ and $N$ is a Nijenhuis operator, we have 
\begin{eqnarray}
 \llbracket N(x), N(y), N(z)\rrbracket&=&N\Big(\llbracket N(x), N(y), z\rrbracket+\llbracket N(x), y, N(z)\rrbracket+\llbracket x, N(y), N(z)\rrbracket\Big)\nonumber\\
 &&\quad-N^2\Big(\llbracket N(x), y, z\rrbracket+\llbracket x, N(y), z\rrbracket+\llbracket x, y, N(z)\rrbracket\Big)+N^3(\llbracket x, y, z\rrbracket),\nonumber\\
 &=&N\Big(\llbracket N(x), N(y), z\rrbracket +\llbracket N(x), y, N(z)]\llbracket+\llbracket x, N(y), N(z)\rrbracket\Big)\nonumber\\
 &&\quad-N\Big(\llbracket N(x), y, z\rrbracket+\llbracket x, N(y), z\rrbracket+\llbracket x, y, N(z)\rrbracket\Big)+N(\llbracket x, y, z\rrbracket),\nonumber\\
&=&N\Big(\llbracket N(x), N(y), z\rrbracket+\llbracket N(x), y, N(z)\rrbracket+\llbracket x, N(y), N(z)\rrbracket\nonumber\\
 &&\quad-\llbracket N(x), y, z\rrbracket +\llbracket x, N(y), z\rrbracket +\llbracket x, y, N(z)\rrbracket +\llbracket x, y, z\rrbracket\Big).\nonumber
\end{eqnarray}
Which means that $N$ is a Rota-Baxter operator of weight $0$. The converse is proved by going up the egalities.
\end{enumerate}
\end{proof}


\subsection{Reynolds operator on compatible ternary Leibniz algebras}
\begin{definition}\label{rbc}
1) A linear map $R: L\rightarrow L$ on a  Leibniz algebra is said to be a {\bf Reynolds operator} on $L$ if
$$[R(x), R(y)])=R\Big([R(x), y]+[x, R(y)]-[R(x), R(y)]),$$
for any $x, y\in L$.\\
2)  A  Reynolds operator on a compatible Leibniz algebra $(L, [-, -]_1, [-, -]_2)$ 
is a Reynolds operator on both $(L, [-, -]_1)$ and $(L, [-, -]_2)$.\\
3)  A linear map $R: L\rightarrow L$ on a ternary Leibniz algebra is called
 a Reynolds operator if 
\begin{eqnarray}
 \llbracket R(x), R(y), R(z)\rrbracket=R\Big(\llbracket R(x), R(y), z\rrbracket+\llbracket R(x), y, R(z)\rrbracket +\llbracket x, R(y), 
R(z)\rrbracket-\llbracket R(x), R(y), R(z)\rrbracket\Big), \label{reti}
\end{eqnarray}
for any $x, y, z\in L$.
\end{definition}

\begin{definition}\label{ret}
Let $(L, \llbracket-, -, -\rrbracket_1, \llbracket-, -, -\rrbracket_2)$ be a compatible ternary Leibniz algebra and $R : L\rightarrow L$ a linear
map. We say that $R$ is a Reynolds operator on the compatible Leibniz algebra $L$ if $R$ is a Reynolds operator on both ternary Leibniz
algebra $(L, \llbracket-, -, -\rrbracket_1)$ and $(L, \llbracket-, -, -\rrbracket_2)$; that is
\begin{eqnarray}
 \llbracket R(x), R(y), R(z)\rrbracket_1=R\Big(\llbracket R(x), R(y), z\rrbracket_1+\llbracket R(x), y, R(z)\rrbracket_1
 +\llbracket x, R(y), R(z)\rrbracket_1
-\llbracket R(x), R(y), R(z)\rrbracket_1\Big), \label{ret1}\\
 \llbracket R(x), R(y), R(z)\rrbracket_2=R\Big(\llbracket R(x), R(y), z\rrbracket_2+\llbracket R(x), y, R(z)\rrbracket_2
 +\llbracket x, R(y), R(z)\rrbracket_2
-\llbracket R(x), R(y), R(z)\rrbracket_2\Big), \label{ret2}
\end{eqnarray}
 for any $x, y, z\in L$.
\end{definition}

The following proposition comes from direct computation.
\begin{proposition}
 Let $P : L\rightarrow L$ be a Reynolds operator on a Leibniz algebra $(L, [-, -])$.
Then, $R$ is a Reynolds operator of the ternary Leibniz algebra $T(L)$.\\
Moreover, if $R$ is also a Reynolds operator on another Leibniz algebra $(L, [-, -]_1)$ which is compatible with $(L, [-, -])$
 such that the associated ternary Leibniz algebras are compatibles,
then, $R$ is a Reynolds operator on the compatible ternary Leibniz algebra $T(L)$.
\end{proposition}


\subsection{Elements of centroid on compatible ternary Leibniz algebras}
\begin{definition}\label{ec}
1) A linear map $\theta : L\rightarrow L$ on a ternary Leibniz algebra $(L, [-, -, -])$  is said to be an element of centroid of $L$ if, 
for any $x, y, z\in L$,
\begin{eqnarray}
 \theta(\llbracket x, y, z\rrbracket)=\llbracket\theta(x), y, z\rrbracket=\llbracket x, \theta(y), z\rrbracket=\llbracket x, y,
 \theta(z)\rrbracket.\label{ec1}
\end{eqnarray}
2) An element of centroid on a compatible ternary Leibniz algebra $(L, [-, -, -]_1, [-, -, -]_2)$ is an element of centroid  on both 
$(L, [-, -, -]_1)$ and $(L, [-, -, -]_2)$.
\end{definition}
\begin{remark}
The sum of two element of centroids is also another one.
\end{remark}
\begin{example}
For any $\lambda\in\mathbb{K}$ and any $x\in L$, the linear map $\theta(x)=\lambda x$ is an element of centroid of $L$.
\end{example}
\begin{example}
 Let $L$ be a ternary Leibniz color algebra and $\theta : L\rightarrow L$ an element of centroid of $L$. Then $L$ becomes a ternary Leibniz
 color algebra in each of the following brackets :
 $$\{x, y, z\}_1=\llbracket \theta(x), y, z\rrbracket, \quad\{x, y, z\}_2=\llbracket\theta(x), \theta(y), z\rrbracket,\quad\{x, y, z\}_3=\llbracket\theta(x), \theta(y), \theta(z)\rrbracket,$$
 for any $x, y, z\in L$.
\end{example}

From a direct computation, we get :
\begin{proposition}
 Let $\theta : L\rightarrow L$ be an element of centroid on a Leibniz algebra $(L, [-, -])$.
Then, $\theta$ is an element of centroid of the ternary Leibniz algebra $T(L)$.\\
Moreover, if $\theta$ is also an element of centroid on another Leibniz algebra $(L, [-, -]_1)$ which is compatible with $(L, [-, -])$
 such that the associated ternary Leibniz algebras are compatibles,
then, $\theta$ is an element of centroid on the compatible ternary Leibniz algebra $T(L)$.
\end{proposition}

\section{Classifications of (compatible) ternary Leibniz algebras}
In this section, we give the classification of Ternary Leibniz algebra and the compatible one in low dimension.
\subsection{Classification of ternary Leibniz algebras}

Let $(L, \llbracket -, - , -\rrbracket)$ be an $n$-dimensional Ternary Leibniz algebra, $\{e_i\}$ be a basis of $L$. For any $i, j, k\in \mathbb{N},
 1\leq i, j, k\leq n$, let us put 
$$\llbracket e_i, e_j, e_k\rrbracket=\sum_{p=1}^{n}\chi_{ijk}^pe_p.$$
The axiom in Definition \ref{dt} is equivalent to
\begin{eqnarray}
\sum_{r=1}^n\Bigg(\chi_{ijk}^r\chi_{rpq}^s-\chi_{kpq}^r\chi_{ijr}^s-\chi_{jpq}^r\chi_{irk}^s-\chi_{ipq}^r\chi_{rjk}^s\Bigg)&=&0.
\end{eqnarray}

\begin{theorem}
The isomorphism class of 2-dimensional ternary Leibniz algebras is given by the following representatives.
\begin{itemize}
\item
$\mathcal{L}_1$ :
 $\begin{array}{ll}  
\llbracket e_1,e_1, e_1\rrbracket=e_2;
\end{array}$

\item
$\mathcal{L}_2$ :
 $\begin{array}{ll}  
\llbracket e_1,e_2, e_2\rrbracket=e_1,\quad \llbracket e_2,e_2, e_2\rrbracket=e_1.
\end{array}$
\end{itemize}
\end{theorem}

\begin{theorem}
The isomorphism class of 3-dimensional ternary Leibniz algebras is given by the following representatives.
\begin{itemize}
\item
$\mathcal{L}_1$ :
 $\begin{array}{ll}  
\llbracket e_1,e_1, e_1\rrbracket=e_3,\quad \llbracket e_1,e_2, e_1\rrbracket=e_3,\quad \llbracket e_2,e_1, e_1\rrbracket=e_3,\quad  \llbracket e_2,e_1, e_2\rrbracket=e_3,\quad  \llbracket e_2,e_2, e_1\rrbracket=e_3,\\  \llbracket e_2,e_2, e_2\rrbracket=\frac{1}{2}e_3;
\end{array}$

\item
$\mathcal{L}_2$ :
 $\begin{array}{ll}  
\llbracket e_1,e_1, e_3\rrbracket=e_2,\quad \llbracket e_1,e_3, e_1\rrbracket=e_2,\quad  \llbracket e_1,e_3, e_3\rrbracket=e_2,\quad 
\llbracket e_3,e_1, e_3\rrbracket=e_2,\quad  \llbracket e_3,e_3, e_3\rrbracket=e_2;
\end{array}$
\item
$\mathcal{L}_3$ :
 $\begin{array}{ll}  
\llbracket e_1,e_2, e_2\rrbracket=e_1+e_3,\quad \llbracket e_2,e_2, e_2\rrbracket=e_1+e_3,\quad  \llbracket e_3,e_2, e_2\rrbracket=e_1+e_3;
\end{array}$
\item
$\mathcal{L}_4^\alpha$ :
 $\begin{array}{ll}  
\llbracket e_2,e_1, e_1\rrbracket=e_2,\quad \llbracket e_2,e_1, e_3\rrbracket=\alpha e_2,\quad  \llbracket e_2,e_3, e_1\rrbracket=e_2,\quad
\llbracket e_2,e_3, e_3\rrbracket=e_2;
\end{array}$
\item
$\mathcal{L}_5$ :
 $\begin{array}{ll}  
\llbracket e_1,e_3, e_3\rrbracket=e_1+e_2,\quad \llbracket e_2,e_3, e_3\rrbracket=e_1+e_2,\quad  \llbracket e_3,e_3, e_3\rrbracket=e_1+e_2;
\end{array}$
\item
$\mathcal{L}_6$ :
 $\begin{array}{ll}  
\llbracket e_2,e_2, e_2\rrbracket=e_1,\,\llbracket e_2,e_2, e_3\rrbracket=e_1,\,  \llbracket e_2,e_3, e_2\rrbracket=e_1,\,
\llbracket e_2,e_3, e_3\rrbracket=e_1,\, \llbracket e_3,e_2, e_3\rrbracket=e_1,\,\llbracket e_3,e_3, e_3\rrbracket=e_1 ;
\end{array}$
\item
$\mathcal{L}_7$ :
 $\begin{array}{ll}  
\llbracket e_2,e_1, e_1\rrbracket=e_2,\quad\llbracket e_2,e_1, e_3\rrbracket=e_2,\quad  \llbracket e_2,e_3, e_1\rrbracket=e_2,\quad
\llbracket e_2,e_3, e_3\rrbracket=e_2;
\end{array}$
\item
$\mathcal{L}_8^\alpha$ :
 $\begin{array}{ll}  
\llbracket e_1,e_3, e_1\rrbracket=e_2,\, \llbracket e_1,e_3, e_3\rrbracket=e_2,\,\,
\llbracket e_3,e_1, e_1\rrbracket=e_2,\,\, \llbracket e_3,e_1, e_3\rrbracket=e_2,\,
\llbracket e_3,e_3, e_1\rrbracket=e_2,\, \llbracket e_3,e_3, e_3\rrbracket=2\alpha e_2.
\end{array}$
\end{itemize}
\end{theorem}

\subsection{Classification of compatible ternary Leibniz algebras}
%

Let $(L, \llbracket-, -, -\rrbracket_1, \llbracket-, -, -\rrbracket_2)$ be an $n$-dimensional compatible Ternary Leibniz algebra, $\{e_i\}$ be a basis of $L$. For any $i, j\in \mathbb{N}, 1\leq i, j, k\leq n$, let us put 
$$\llbracket e_i, e_j, e_k\rrbracket_1=\sum_{p=1}^{n}\chi_{ijk}^pe_p, 
\quad\llbracket e_i, e_j, e_k\rrbracket_2=\sum_{p=1}^{n}\gamma_{ijk}^pe_p $$
The axioms in Definition \ref{dc} are respectively equivalent to

\begin{eqnarray}
\sum_{r=1}^n\Bigg(\chi_{ijk}^r\chi_{rpq}^s-\chi_{kpq}^r\chi_{ijr}^s-\chi_{jpq}^r\chi_{irk}^s-\chi_{ipq}^r\chi_{rjk}^s\Bigg)&=&0.
\end{eqnarray}
\begin{eqnarray}
\sum_{r=1}^n\Bigg(\gamma_{ijk}^r\gamma_{rpq}^s-\gamma_{kpq}^r\gamma_{ijr}^s-\gamma_{jpq}^r\gamma_{irk}^s-\gamma_{ipq}^r\gamma_{rjk}^s\Bigg)&=&0.
\end{eqnarray}

\begin{eqnarray}
\sum_{r=1}^n\Bigg(\chi_{ijk}^r\gamma_{rpq}^s+\gamma_{ijk}^r\chi_{rpq}^s\Bigg)=\sum_{r=1}^n\Bigg(\chi_{kpq}^r\gamma_{ijr}^s+\gamma_{kpq}^r\chi_{ijr}^s+\chi_{jpq}^r\gamma_{irk}^s+\gamma_{jpq}^r\gamma_{irk}^s+\chi_{ipq}^r\gamma_{rjk}^s+\gamma_{ipq}^r\chi_{rjk}^s\Bigg).
\end{eqnarray}

\begin{theorem}
The isomorphism class of 2-dimensional compatible ternary Leibniz algebras is given by the following representatives.
\begin{itemize}
\item
$\mathcal{L}_1$ :
 $\begin{array}{ll}  
\llbracket e_1,e_1, e_1\rrbracket_1=e_2,\quad \llbracket e_2,e_1, e_1\rrbracket_2=e_2;
\end{array}$

\item
$\mathcal{L}_2^\alpha$ :
 $\begin{array}{ll}  
\llbracket e_1,e_1, e_1\rrbracket_1=e_2,\quad \llbracket e_2,e_1, e_1\rrbracket_1=\alpha e_2,\quad \llbracket e_1,e_1, e_1\rrbracket_2= e_2;
\end{array}$

\item
$\mathcal{L}_3$ :
 $\begin{array}{ll}  
\llbracket e_1,e_2, e_2\rrbracket_1=e_1,\quad \llbracket e_1,e_2, e_2\rrbracket_2=e_1,\quad \llbracket e_2,e_2, e_2\rrbracket_2=e_1.
\end{array}$
\end{itemize}
\end{theorem}

{\footnotesize
\begin{theorem}
The isomorphism class of 3-dimensional compatible ternary Leibniz algebras is given by the following representatives.
\begin{itemize}
\item
$\mathcal{L}_1$ :
 $\begin{array}{ll}  
\llbracket e_1,e_1, e_1\rrbracket_1=e_3,\,\llbracket e_2,e_1, e_1\rrbracket_1=e_3,\, \llbracket e_3,e_1, e_1\rrbracket_1=e_3,
\,\llbracket e_1,e_1, e_1\rrbracket_2=e_3,\,\llbracket e_2,e_1, e_1\rrbracket_2=e_2,\,\llbracket e_3,e_1, e_1\rrbracket_2=e_2;
\end{array}$
\item

$\mathcal{L}_2^\alpha $ :
 $\begin{array}{ll}  
\llbracket e_1,e_2, e_1\rrbracket_1=e_3,\quad \llbracket e_1,e_2, e_2\rrbracket_1=\alpha e_3,\quad \llbracket e_2,e_2, e_1\rrbracket_1=e_3,\quad \llbracket e_1,e_2, e_1\rrbracket_2=e_3,\quad \llbracket e_2,e_1, e_1\rrbracket_2=e_3,\\ \llbracket e_2,e_1, e_2\rrbracket_2=e_3,\quad \llbracket e_2,e_2, e_2\rrbracket_2=e_3;
\end{array}$

\item
$\mathcal{L}_3^\alpha$ :
 $\begin{array}{ll}  
\llbracket e_1,e_3, e_3\rrbracket_1=2\alpha e_1,\quad \llbracket e_2,e_3, e_3\rrbracket_1=e_2,\quad  \llbracket e_3,e_3, e_3\rrbracket_1=e_2,\quad \llbracket e_1,e_3, e_3\rrbracket_2=e_2,\quad \llbracket e_3,e_3, e_3\rrbracket_2=e_2;
\end{array}$

\item
$\mathcal{L}_4$ :
 $\begin{array}{ll}  
\llbracket e_1,e_3, e_1\rrbracket_1=e_2,\,\llbracket e_3,e_1, e_1\rrbracket_1=e_2,\,\llbracket e_3,e_1, e_3\rrbracket_1=\frac{1}{2}e_2,\, \llbracket e_1,e_1, e_3\rrbracket_2=e_2,\,
\llbracket e_1,e_3, e_3\rrbracket_2=e_2,\, \llbracket e_3,e_1, e_3\rrbracket_1=e_2,\\
\llbracket e_3,e_3, e_3\rrbracket_2=e_2;
\end{array}$

\item
$\mathcal{L}_5$ :
 $\begin{array}{ll}  
\llbracket e_2,e_2, e_1\rrbracket_1=e_1,\,\llbracket e_2,e_2, e_2\rrbracket_1=e_3,\,\llbracket e_3,e_2, e_2\rrbracket_1=-e_1+e_3,\,\llbracket e_1,e_2, e_2\rrbracket_2=e_3,\, \llbracket e_2,e_2, e_2\rrbracket_2=e_3,\,\llbracket e_3,e_2, e_2\rrbracket_2=e_3;
\end{array}$

\item
$\mathcal{L}_6^\alpha$ :
 $\begin{array}{ll}  
\llbracket e_2,e_1, e_3\rrbracket_1=e_2,\qquad \llbracket e_2,e_3, e_1\rrbracket_1=e_2,\qquad \llbracket e_2,e_3, e_3\rrbracket_1=e_2,\qquad\llbracket e_2,e_1, e_1\rrbracket_2=e_2,\qquad\llbracket e_2,e_3, e_3\rrbracket_2=\alpha e_2;
\end{array}$
\item
$\mathcal{L}_7$ :
 $\begin{array}{ll}  
\llbracket e_1,e_3, e_3\rrbracket_1=e_1,\,\llbracket e_2,e_3, e_3\rrbracket_1=e_1,\, \llbracket e_3,e_3, e_3\rrbracket_1=e_2,
\,\llbracket e_1,e_3, e_3\rrbracket_2=e_2,\,\llbracket e_2,e_3, e_3\rrbracket_2=e_1+e_2,\,\llbracket e_3,e_3, e_3\rrbracket_2=e_1+e_2;
\end{array}$

\item
$\mathcal{L}_8$ :
 $\begin{array}{ll}  
\llbracket e_1,e_3, e_1\rrbracket_1=e_2,\,\llbracket e_1,e_3, e_3\rrbracket_1=e_2,\, \llbracket e_3,e_1, e_1\rrbracket_1=e_2,
\,\llbracket e_1,e_1, e_1\rrbracket_2=e_2,\, \llbracket e_1,e_3, e_3\rrbracket_2=e_2,\,
\llbracket e_3,e_1, e_3\rrbracket_2=e_2,\,\llbracket e_3,e_3, e_3\rrbracket_2=e_2;
\end{array}$

\item
$\mathcal{L}_9^\alpha$ :
 $\begin{array}{ll}  
\llbracket e_1,e_2, e_2\rrbracket_1=e_1-e_3,\,\llbracket e_2,e_2, e_2\rrbracket_1=e_1,\,\llbracket e_3,e_2, e_2\rrbracket_1=e_1,\,\llbracket e_1,e_2, e_2\rrbracket_2=e_1+\alpha e_3,\, \llbracket e_2,e_2, e_2\rrbracket_2=e_1+e_3,\,\llbracket e_3,e_2, e_2\rrbracket_2=e_3.
\end{array}$
\end{itemize}
\end{theorem}}


\section{Classification of operators on (compatible) ternary Leibniz algebras}
This section is dedicated to the looking for operators on the algebras found in the previous section.

\subsection{Derivations}
Let $(\mathcal{L}, \llbracket -, -, -\rrbracket)$ be an $n$-dimensional ternary Leibniz algebra with basis $\{e_i\} ( 1\leq i \leq n)$ and 
let $d$ be a derivation on $\mathcal{L}$.
 For any $i, j, k, p\in \mathbb{N}, 1\leq i, j, k, p\leq n$, let us put 
$$\llbracket e_i, e_j, e_k\rrbracket=\sum_{p=1}^{n}\gamma_{ijk}^pe_p \quad\mbox{and}\quad
  D(e_i) =d_{1i}e_1 +d_{2i}e_2 + \cdots + d_{ni}e_n,$$
Then, in term of basis elements, equation (\ref{dev}) is equivalent to :
\begin{eqnarray}
 \sum_{r=1}^n\gamma_{ijk}^rd_{sr}&=&\sum_{r=1}^n\Bigg(d_{ri}\gamma_{rjk}^s+d_{rj}\gamma_{irk}^s+d_{rk}\gamma_{ijr}^s\Bigg)
+\lambda\sum_{q=1}^n\sum_{r=1}^n\Bigg(d_{qi}d_{rj}\gamma_{qrk}^s+d_{qj}d_{rk}\gamma_{iqr}^s+d_{qi}d_{rk}\gamma_{qjr}^s\Bigg)\nonumber\\
&&+\lambda^2\sum_{p=1}^n\sum_{q=1}^n\sum_{r=1}^nd_{pi}d_{qj}d_{rk}\gamma_{pqr}^s,\nonumber
\end{eqnarray} for $i,j,k,s=1,2,\dots,n$.

\begin{proposition}
The description of the derivation of every 2-dimensional Ternary Leibniz algebra is given below.
\begin{itemize}
    \item[$\mathcal{L}_1$ : ]
    $\left(\begin{array}{ccc}
0&0\\
d_{21}&0
\end{array}
\right)
;
\qquad \mathcal{L}_2 :
    \left(\begin{array}{ccc}
d_{11}&d_{11}\\
0&0
\end{array}
\right)
.$
\end{itemize}
\end{proposition}

\begin{proposition}
The description of the derivation of every 3-dimensional Ternary Leibniz algebra is given below.
\begin{itemize}
    \item[$\mathcal{L}_1$ : ]
    $
\left(\begin{array}{ccc}
d_{11}&0&d_{13}\\
0&d_{22}&0\\
d_{31}&0&-d_{13}
\end{array}
\right);\quad
\mathcal{L}_2 : 
\left(\begin{array}{ccc}
0&0&0\\
d_{21}&0&d_{23}\\
0&0&0\\
\end{array}
\right);\quad
\mathcal{L}_3 :
\left(\begin{array}{ccc}
d_{33}&d_{12}&d_{13}\\
0&0&0\\
d_{31}&d_{13}-d_{12}+d_{33}&d_{33}
\end{array}
\right);
\quad\mathcal{L}_4 : 
\left(\begin{array}{ccc}
0&0&0\\
0&d_{22}&0\\
0&0&0\\
\end{array}
\right);
$

\item[$\mathcal{L}_5$ : ]
    $
\left(\begin{array}{ccc}
d_{11}&d_{12}&d_{13}\\
d_{21}&d_{11}&d_{11}+d_{12}-d_{13}\\
0&0&0
\end{array}
\right);\quad
\mathcal{L}_6 : 
\left(\begin{array}{ccc}
0&d_{12}&d_{13}\\
0&0&0\\
0&0&0\\
\end{array}
\right);\quad
\mathcal{L}_7 :
\left(\begin{array}{ccc}
d_{11}&0&d_{13}\\
0&d_{22}&0\\
-d_{11}&0&-d_{13}
\end{array}
\right);
\quad\mathcal{L}_8 : 
\left(\begin{array}{ccc}
0&0&0\\
d_{21}&0&d_{23}\\
0&0&0\\
\end{array}
\right).
$

\end{itemize}
\end{proposition}

\begin{corollary}\,
\begin{itemize}
	\item The dimensions of the derivation of ternary Leibniz algebras of $2$ dimensional are $1$.
	\item The dimensions of the derivation of ternary Leibniz algebras of $3$ dimensional range between $1$ and $4$.
	\end{itemize}
\end{corollary}
Let $(L, \llbracket-, -, -\rrbracket_1, \llbracket-, -, -\rrbracket_2)$ be an $n$-dimensional compatible Ternary Leibniz algebra, $\{e_i\}$ be 
a basis of $L$ and $d$ a derivation on $L$. For any $i, j\in \mathbb{N}, 1\leq i, j, k\leq n$, let us put 
$$\llbracket e_i, e_j, e_k\rrbracket_1=\sum_{p=1}^{n}\chi_{ijk}^pe_p, 
\quad
\llbracket e_i, e_j, e_k\rrbracket_2=\sum_{p=1}^{n}\gamma_{ijk}^pe_p, 
\quad 
 d(e_i) =d_{1i}e_1 +d_{2i}e_2 + \cdots + d_{ni}e_n.
$$
The axioms in [$4)$, Definition \ref{dlt}] is equivalent to
\begin{eqnarray}
 \sum_{r=1}^n\chi_{ijk}^rd_{sr}&=&\sum_{r=1}^n\Bigg(d_{ri}\chi_{rjk}^s+d_{rj}\chi_{irk}^s+d_{rk}\chi_{ijr}^s\Bigg)
+\lambda\sum_{q=1}^n\sum_{r=1}^n\Bigg(d_{qi}d_{rj} \chi_{qrk}^s+d_{qj}d_{rk}\chi_{iqr}^s+d_{qi}d_{rk}\chi_{qjr}^s\Bigg)\nonumber\\
&&+\lambda^2\sum_{p=1}^n\sum_{q=1}^n\sum_{r=1}^nd_{pi}d_{qj}d_{rk}\chi_{pqr}^s,
\end{eqnarray} 
\begin{eqnarray}
 \sum_{r=1}^n\gamma_{ijk}^rd_{sr}&=&\sum_{r=1}^n\Bigg(d_{ri}\gamma_{rjk}^s+d_{rj}\gamma_{irk}^s+d_{rk}\gamma_{ijr}^s\Bigg)
+\lambda\sum_{q=1}^n\sum_{r=1}^n\Bigg(d_{qi}d_{rj}\gamma_{qrk}^s+d_{qj}d_{rk}\gamma_{iqr}^s+d_{qi}d_{rk}\gamma_{qjr}^s\Bigg)\nonumber\\
&&+\lambda^2\sum_{p=1}^n\sum_{q=1}^n\sum_{r=1}^nd_{pi}d_{qj}d_{rk}\gamma_{pqr}^s,
\end{eqnarray} for $i,j,k,s=1,2,\dots,n$.

\begin{proposition}
The description of the derivation of every 2-dimensional Compatible Ternary Leibniz algebra is given below.
\begin{itemize}
    \item[$\mathcal{L}_1$ : ]
    $\left(\begin{array}{ccc}
0&0\\
0&0\\
\end{array}
\right);
\qquad\mathcal{L}_2 :
\left(\begin{array}{ccc}
0&0\\
0&0\\
\end{array}
\right)
;
\qquad\mathcal{L}_3 :
\left(\begin{array}{ccc}
0&0\\
0&0\\
\end{array}
\right).
$
\end{itemize}
\end{proposition}

\begin{proposition}
The description of the derivation of every 3-dimensional Compatible Ternary Leibniz algebra is given below.
\begin{itemize}
    \item[$\mathcal{L}_1$ : ]
    $\left(\begin{array}{ccc}
0&0&0\\
0&0&0\\
0&0&0\\
\end{array}
\right);
\qquad \mathcal{L}_2 : 
\left(\begin{array}{ccc}
0&0&0\\
0&0&0\\
d_{31}&d_{32}&0
\end{array}
\right);
\qquad\mathcal{L}_3 : 
\left(\begin{array}{ccc}
0&0&0\\
0&0&0\\
0&0&0\\
\end{array}
\right);\,\,
\mathcal{L}_4 : 
\left(\begin{array}{ccc}
0&0&0\\
d_{21}&0&d_{23}\\
0&0&0\\
\end{array}
\right);
\qquad\mathcal{L}_5 : 
\left(\begin{array}{ccc}
0&0&0\\
0&0&0\\
0&0&0\\
\end{array}
\right)
$

\item[$\mathcal{L}_6$ : ]
    $\left(\begin{array}{ccc}
0&0&0\\
0&0&0\\
0&0&0\\
\end{array}
\right);
\qquad\mathcal{L}_7 : 
    \left(\begin{array}{ccc}
0&0&0\\
0&0&0\\
0&0&0\\
\end{array}
\right);
;
\quad\mathcal{L}_8 : 
    \left(\begin{array}{ccc}
0&0&0\\
d_{21}&0&d_{23}\\
0&0&0
\end{array}
\right);
\qquad\mathcal{L}_9 : 
\left(\begin{array}{ccc}
0&0&0\\
0&0&0\\
0&0&0
\end{array}
\right).
$
\end{itemize}
\end{proposition}

\begin{corollary}\,
\begin{itemize}
	\item The dimension of the derivations of Compatible Ternary Leibniz algebras of $2$-dimensional  is $0$
	\item The dimensions of the derivations of Compatible Ternary Leibniz algebras of $3$-dimensional range between $0$ and $2$.
	\end{itemize}
\end{corollary}

\subsection{Averaging operators}
Let $(\mathcal{L}, \llbracket -, -, -\rrbracket)$ be an $n$-dimensional ternary Leibniz algebra with basis $\{e_i\} ( 1\leq i \leq n)$ and 
let $\vartheta$ be an averaging operator on $\mathcal{L}$.
 For any $i, j, k, p\in \mathbb{N}, 1\leq i, j, k, p\leq n$, let us put 
$$\llbracket e_i, e_j, e_k\rrbracket=\sum_{p=1}^{n}\gamma_{ijk}^pe_p \quad\mbox{and}\quad
  \vartheta(e_i) = \vartheta_{1i}e_1 + \vartheta_{2i}e_2 + \cdots + \vartheta_{ni}e_n,$$
Then, in term of basis elements, equation (\ref{av1}) is equivalent to :
\begin{eqnarray}
\sum_{r=1}^n \sum_{p=1}^n\sum_{q=1}^n\vartheta_{pi}\vartheta_{qj}\vartheta_{rj}\gamma_{pqr}^s
=\sum_{r=1}^n \sum_{p=1}^n\sum_{q=1}^n\vartheta_{pi}\vartheta_{qj}\gamma_{pqk}^r\vartheta_{sr}\!
=\sum_{r=1}^n \sum_{p=1}^n\sum_{q=1}^n\vartheta_{pi}\vartheta_{qk}\gamma_{pjq}^r\vartheta_{rs}
=\sum_{r=1}^n \sum_{p=1}^n\sum_{q=1}^n\vartheta_{pi}\vartheta_{qk}\gamma_{ipq}^r\vartheta_{sr},\nonumber
\end{eqnarray} for $i,j,k,s=1,2,\dots,n$.

\begin{proposition}
The description of the Averaging operator of every 2-dimensional Ternary Leibniz algebra is given below.
\begin{itemize}
    \item[$\mathcal{L}_1$ : ]
    $
 \left(\begin{array}{ccc}
\vartheta_{11}&0\\
\vartheta_{21}&\vartheta_{11}
\end{array}
\right),\quad
\left(\begin{array}{ccc}
0&0\\
\vartheta_{21}&\vartheta_{22}
\end{array}
\right)
;\quad\mathcal{L}_2 :    
 \left(\begin{array}{ccc}
\vartheta_{11}&0\\
0&\vartheta_{11}
\end{array}
\right),
\quad
\left(\begin{array}{ccc}
\vartheta_{11}&\vartheta_{12}\\
0&0\\
\end{array}
\right)
,
\quad
\left(\begin{array}{ccc}
0&\vartheta_{12}\\
0&-\vartheta_{12}\\
\end{array}
\right)
.$
\end{itemize}
\end{proposition}

\begin{proposition}
The description of the Averaging operator of every 3-dimensional Ternary Leibniz algebra is given below.
\begin{itemize}
    \item[$\mathcal{L}_1$ : ]
    $\left(\begin{array}{ccc}
\vartheta_{11}&0&0\\
0&\vartheta_{11}&0\\
0&0&\vartheta_{11}
\end{array}
\right),\quad
\left(\begin{array}{ccc}
0&0&0\\
0&0&0\\
\vartheta_{31}&\vartheta_{32}&\vartheta_{33}
\end{array}
\right)\quad
\left(\begin{array}{ccc}
\vartheta_{11}&0&0\\
0&\vartheta_{11}&0\\
\vartheta_{31}&\vartheta_{32}&\vartheta_{11}
\end{array}
\right)
;$

\item[$\mathcal{L}_2$ : ]
    $\left(\begin{array}{ccc}
\vartheta_{33}&0&0\\
\vartheta_{21}&\vartheta_{33}&\vartheta_{23}\\
0&0&\vartheta_{33}
\end{array}
\right),\quad\left(\begin{array}{ccc}
\vartheta_{11}&0&\vartheta_{13}\\
\vartheta_{21}&0&\vartheta_{23}\\
0&0&0
\end{array}
\right),\quad\left(\begin{array}{ccc}
0&0&0\\
\vartheta_{21}&\vartheta_{22}&\vartheta_{23}\\
0&0&0
\end{array}
\right),\quad\left(\begin{array}{ccc}
0&0&\vartheta_{13}\\
\vartheta_{21}&0&\vartheta_{23}\\
0&0&0
\end{array}
\right)
;$

\item[$\mathcal{L}_3$ : ]
    $\left(\begin{array}{ccc}
\vartheta_{22}&0&0\\
0&\vartheta_{22}&0\\
0&0&\vartheta_{22}
\end{array}
\right),\quad
\left(\begin{array}{ccc}
\vartheta_{11}&0&\vartheta_{13}\\
0&0&0\\
0&0&0
\end{array}
\right),\quad
\left(\begin{array}{ccc}
0&0&0\\
0&0&0\\
0&\vartheta_{32}
\end{array}
\right)
;$

\item[$\mathcal{L}_4$ : ]
    $\left(\begin{array}{ccc}
\vartheta_{22}&0&0\\
0&\vartheta_{22}&0\\
0&0&\vartheta_{22}
\end{array}
\right),\quad
\left(\begin{array}{ccc}
0&0&0\\
\vartheta_{21}&\vartheta_{22}&\vartheta_{23}\\
0&0&0
\end{array}
\right),\quad
\left(\begin{array}{ccc}
0&0&0\\
0&\vartheta_{22}&0\\
0&0&0
\end{array}
\right),\quad
\left(\begin{array}{ccc}
\vartheta_{11}&0&\vartheta_{13}\\
0&0&0\\
\vartheta_{31}&0&\vartheta_{33}
\end{array}
\right)
;$

\item[$\mathcal{L}_5$ : ]
    $\left(\begin{array}{ccc}
\vartheta_{11}&\vartheta_{12}&\vartheta_{13}\\
\vartheta_{21}&\vartheta_{22}&\vartheta_{23}\\
0&0&0
\end{array}
\right),\quad
\left(\begin{array}{ccc}
\vartheta_{11}&-\vartheta_{11}&\vartheta_{13}\\
\vartheta_{21}&-\vartheta_{21}&\vartheta_{23}\\
-\vartheta_{11}-\vartheta_{21}&-\vartheta_{11}+\vartheta_{21}&-\vartheta_{13}\vartheta_{23}\\
\end{array}
\right)
,\quad
\left(\begin{array}{ccc}
\vartheta_{11}&\vartheta_{12}&\vartheta_{13}\\
\vartheta_{21}&\vartheta_{11}&-\vartheta_{13}\\
0&0&\vartheta_{11}+\vartheta_{12}
\end{array}
\right)
;$

\item[$\mathcal{L}_6$ : ]
    $\left(\begin{array}{ccc}
\vartheta_{33}&0&0\\
0&\vartheta_{33}&0\\
0&0&\vartheta_{33}
\end{array}
\right),\quad
\left(\begin{array}{ccc}
\vartheta_{33}&\vartheta_{12}&\vartheta_{13}\\
0&\vartheta_{33}&0\\
0&0&\vartheta_{33}
\end{array}
\right),\quad
\left(\begin{array}{ccc}
\vartheta_{11}&\vartheta_{12}&\vartheta_{13}\\
0&0&0\\
0&0&0
\end{array}
\right),\quad
\left(\begin{array}{ccc}
0&-\vartheta_{32}&0\\
0&-\vartheta_{32}&0\\
0&\vartheta_{23}&0
\end{array}
\right)
;$

\item[$\mathcal{L}_7$ : ]
    $\left(\begin{array}{ccc}
\vartheta_{11}&\vartheta_{12}&\vartheta_{13}\\
\vartheta_{21}&\vartheta_{22}&\vartheta_{23}\\
-\vartheta_{11}&-\vartheta_{12}&-\vartheta_{13}
\end{array}
\right),\quad
\left(\begin{array}{ccc}
\vartheta_{11}&0&\vartheta_{13}\\
0&\vartheta_{22}&0\\
\vartheta_{22}-\vartheta_{11}&0&\vartheta_{22}-\vartheta_{13}
\end{array}
\right),\quad
\left(\begin{array}{ccc}
\vartheta_{11}&0&\vartheta_{13}\\
0&0&0\\
\vartheta_{31}&0&\vartheta_{33}
\end{array}
\right)
;$

\item[$\mathcal{L}_8$ : ]
    $\left(\begin{array}{ccc}
\vartheta_{33}&0&0\\
\vartheta_{21}&\vartheta_{33}&\vartheta_{23}\\
0&0&\vartheta_{33}
\end{array}
\right),\quad
\left(\begin{array}{ccc}
\vartheta_{11}&0&0\\
\vartheta_{21}&0&\vartheta_{23}\\
0&0&0
\end{array}
\right)
,\quad
\left(\begin{array}{ccc}
\vartheta_{11}&0&\vartheta_{13}\\
\vartheta_{21}&0&\vartheta_{23}\\
0&0&0
\end{array}
\right),\quad
\left(\begin{array}{ccc}
0&0&0\\
\vartheta_{21}&\vartheta_{22}&\vartheta_{23}\\
0&0&0
\end{array}
\right)
.$
\end{itemize}
\end{proposition}
Let $(L, \llbracket-, -, -\rrbracket_1, \llbracket-, -, -\rrbracket_2)$ be an $n$-dimensional compatible Ternary Leibniz algebra, $\{e_i\}$ be 
a basis of $L$ and $d$ a derivation on $L$. For any $i, j\in \mathbb{N}, 1\leq i, j, k\leq n$, let us put 
$$\llbracket e_i, e_j, e_k\rrbracket_1=\sum_{p=1}^{n}\chi_{ijk}^pe_p, 
\quad
\llbracket e_i, e_j, e_k\rrbracket_2=\sum_{p=1}^{n}\gamma_{ijk}^pe_p, 
\quad 
 \vartheta(e_i) = \vartheta_{1i}e_1 + \vartheta_{2i}e_2 + \cdots + \vartheta_{ni}e_n.
$$
The axioms in [$2)$, Definition \ref{av}] are equivalent to
\begin{align*}
   \sum_{r=1}^n \sum_{p=1}^n\sum_{q=1}^n\vartheta_{pi}\vartheta_{qj}\vartheta_{rj}\chi_{pqr}^s=\sum_{r=1}^n \sum_{p=1}^n\sum_{q=1}^n\vartheta_{pi}\vartheta_{qj}\chi_{pqk}^r\vartheta_{sr}\!=\sum_{r=1}^n \sum_{p=1}^n\sum_{q=1}^n\vartheta_{pi}\vartheta_{qk}\chi_{pjq}^r\vartheta_{rs}=\sum_{r=1}^n \sum_{p=1}^n\sum_{q=1}^n\vartheta_{pi}\vartheta_{qk}\chi_{ipq}^r\vartheta_{sr},&& 
\end{align*} 
and
\begin{align*}
   \sum_{r=1}^n \sum_{p=1}^n\sum_{q=1}^n\vartheta_{pi}\vartheta_{qj}\vartheta_{rj}\gamma_{pqr}^s=\sum_{r=1}^n \sum_{p=1}^n\sum_{q=1}^n\vartheta_{pi}\vartheta_{qj}\gamma_{pqk}^r\vartheta_{sr}\!=\sum_{r=1}^n \sum_{p=1}^n\sum_{q=1}^n\vartheta_{pi}\vartheta_{qk}\gamma_{pjq}^r\vartheta_{rs}=\sum_{r=1}^n \sum_{p=1}^n\sum_{q=1}^n\vartheta_{pi}\vartheta_{qk}\gamma_{ipq}^r\vartheta_{sr},&& 
\end{align*}
for $i,j,k,s=1,2,\dots,n$.

\begin{proposition}
The description of the Averaging operator of every 2-dimensional Compatible Ternary Leibniz algebra is given below.
\begin{itemize}
    \item[$\mathcal{L}_1$ : ]
    $
 \left(\begin{array}{ccc}
0&0\\
\vartheta_{21}&\vartheta_{22}
\end{array}
\right),\quad
\left(\begin{array}{ccc}
\vartheta_{11}&0\\
0&0
\end{array}
\right),\quad
\left(\begin{array}{ccc}
\vartheta_{11}&0\\
0&\vartheta_{11}
\end{array}
\right)
;\quad \mathcal{L}_2 : 
 \left(\begin{array}{ccc}
\vartheta_{11}&0\\
\vartheta_{21}&\vartheta_{11}
\end{array}
\right),\quad
\left(\begin{array}{ccc}
0&0\\
\vartheta_{21}&\vartheta_{22}
\end{array}
\right)
;\quad\mathcal{L}_3 : 
 \left(\begin{array}{ccc}
\vartheta_{11}&\vartheta_{12}\\
0&0
\end{array}
\right),\quad\left(\begin{array}{ccc}
\vartheta_{11}&0\\
0&\vartheta_{11}
\end{array}
\right)
.$

\end{itemize}
\end{proposition}

\begin{proposition}
The description of the Averaging operator of every 3-dimensional Compatible Ternary Leibniz algebra is given below.
\begin{itemize}
    \item[$\mathcal{L}_1$ : ]
    $
 \left(\begin{array}{ccc}
0&0&0\\
\vartheta_{21}&\vartheta_{22}&0\\
\vartheta_{31}&\vartheta_{32}&0\\
\end{array}
\right),\quad
\left(\begin{array}{ccc}
\vartheta_{11}&0&0\\
\vartheta_{21}&\vartheta_{11}&0\\
-\vartheta_{21}&0&\vartheta_{11}\\
\end{array}
\right)
;$
 \item[$\mathcal{L}_2$ : ]
    $
 \left(\begin{array}{ccc}
\vartheta_{11}&\vartheta_{12}&0\\
0&0&0\\
\vartheta_{31}&\vartheta_{32}&0\\
\end{array}
\right),\quad
\left(\begin{array}{ccc}
0&0&0\\
0&0&0\\
\vartheta_{31}&\vartheta_{32}&\vartheta_{33}\\
\end{array}
\right)
,\quad
\left(\begin{array}{ccc}
\vartheta_{22}&0&0\\
0&\vartheta_{22}&0\\
\vartheta_{31}&\vartheta_{32}&\vartheta_{22}\\
\end{array}
\right),\quad
\left(\begin{array}{ccc}
0&0&0\\
0&0&0\\
\vartheta_{31}&\vartheta_{32}&0\\
\end{array}
\right)
;$
 \item[$\mathcal{L}_3$ : ]
    $
 \left(\begin{array}{ccc}
\vartheta_{33}&0&0\\
0&\vartheta_{33}&0\\
0&0&\vartheta_{33}
\end{array}
\right),\quad
\left(\begin{array}{ccc}
\vartheta_{11}&\vartheta_{12}&\vartheta_{13}\\
\vartheta_{21}&\vartheta_{22}&\vartheta_{23}\\
0&0&0
\end{array}
\right)
;$

\item[$\mathcal{L}_4$ : ]
    $\left(\begin{array}{ccc}
\vartheta_{11}&0&\vartheta_{13}\\
\vartheta_{21}&0&\vartheta_{23}\\
0&0&0
\end{array}
\right),\quad
\left(\begin{array}{ccc}
0&0&\vartheta_{13}\\
\vartheta_{21}&0&\vartheta_{23}\\
0&0&0\\
\end{array}
\right),
\quad
\left(\begin{array}{ccc}
0&0&0\\
\vartheta_{21}&\vartheta_{22}&\vartheta_{23}\\
0&0&0
\end{array}
\right),
\quad
\left(\begin{array}{ccc}
\vartheta_{11}&0&0\\
0&\vartheta_{11}&\vartheta_{23}\\
0&0&\vartheta_{11}
\end{array}
\right)
;$
\item[$\mathcal{L}_5$ : ]
    $\left(\begin{array}{ccc}
\vartheta_{11}&\vartheta_{12}&\vartheta_{13}\\
0&0&0\\
\vartheta_{31}&\vartheta_{32}&\vartheta_{33}\\
\end{array}
\right),
\quad
\left(\begin{array}{ccc}
0&0&0\\
0&\vartheta_{22}&0\\
\vartheta_{31}&-\vartheta_{22}&\vartheta_{22}
\end{array}
\right),
\quad
\left(\begin{array}{ccc}
\vartheta_{22}&0&0\\
0&\vartheta_{22}&0\\
0&0&\vartheta_{22}
\end{array}
\right)
;$

\item[$\mathcal{L}_6$ : ]
    $\left(\begin{array}{ccc}
\vartheta_{11}&0&\vartheta_{13}\\
0&0&0\\
\vartheta_{31}&0&\vartheta_{33}\\
\end{array}
\right),\quad
\left(\begin{array}{ccc}
0&0&0\\
\vartheta_{21}&\vartheta_{22}&\vartheta_{23}\\
0&0&0
\end{array}
\right),
\quad
\left(\begin{array}{ccc}
0&0&0\\
\vartheta_{21}&0&\vartheta_{23}\\
0&0&0
\end{array}
\right),
\quad
\left(\begin{array}{ccc}
\vartheta_{33}&0&0\\
0&\vartheta_{33}&0\\
0&0&\vartheta_{33}
\end{array}
\right)
;$

\item[$\mathcal{L}_7$ : ]
    $\left(\begin{array}{ccc}
\vartheta_{11}&0&0\\
0&\vartheta_{11}&0\\
0&0&0
\end{array}
\right),\quad
\left(\begin{array}{ccc}
0&0&0\\
0&\vartheta_{22}&0\\
0&0&0
\end{array}
\right),
\quad
\left(\begin{array}{ccc}
\vartheta_{11}&0&0\\
0&0&0\\
0&0&\vartheta_{11}
\end{array}
\right),
\quad
\left(\begin{array}{ccc}
\vartheta_{11}&0&0\\
0&\vartheta_{11}&0\\
0&0&\vartheta_{11}
\end{array}
\right)
;$

\item[$\mathcal{L}_8$ : ]
    $\left(\begin{array}{ccc}
0&0&0\\
\vartheta_{21}&\vartheta_{22}&\vartheta_{23}\\
0&0&0\\
\end{array}
\right),
\quad\left(\begin{array}{ccc}
0&0&0\\
\vartheta_{21}&0&\vartheta_{23}\\
0&0&0\\
\end{array}
\right),
\quad
\left(\begin{array}{ccc}
\vartheta_{33}&0&0\\
0&\vartheta_{33}&0\\
0&0&\vartheta_{33}
\end{array}
\right)
;$
\item[$\mathcal{L}_9$ : ]
    $\left(\begin{array}{ccc}
\vartheta_{11}&\vartheta_{12}&\vartheta_{13}\\
0&0&0\\
\vartheta_{31}&\vartheta_{32}&\vartheta_{33}\\
\end{array}
\right),
\quad
\left(\begin{array}{ccc}
\vartheta_{22}&0&0\\
0&\vartheta_{22}&0\\
0&0&\vartheta_{22}
\end{array}
\right)
.$
\end{itemize}
\end{proposition}
\subsection{Rota-Baxter operators} 

Let $(\mathcal{L}, \llbracket -, -, -\rrbracket)$ be an $n$-dimensional ternary Leibniz algebra with basis $\{e_i\} ( 1\leq i \leq n)$ and 
let $d$ be a derivation on $\mathcal{L}$.
 For any $i, j, k, p\in \mathbb{N}, 1\leq i, j, k, p\leq n$, let us put 
$$\llbracket e_i, e_j, e_k\rrbracket=\sum_{p=1}^{n}\gamma_{ijk}^pe_p \quad\mbox{and}\quad
  R(e_i) =r_{1i}e_1 +r_{2i}e_2 + \cdots + r_{ni}e_n,$$
Then, in term of basis elements, equation (\ref{rotid}) is equivalent to :
\begin{eqnarray}
 \sum_{r=1}^n\gamma_{ijk}^rd_{sr}&=&\sum_{r=1}^n\Bigg(d_{ri}\gamma_{rjk}^s+d_{rj}\gamma_{irk}^s+d_{rk}\gamma_{ijr}^s\Bigg)
+\lambda\sum_{q=1}^n\sum_{r=1}^n\Bigg(d_{qi}d_{rj}\gamma_{qrk}^s+d_{qj}d_{rk}\gamma_{iqr}^s+d_{qi}d_{rk}\gamma_{qjr}^s\Bigg)\nonumber\\
&&+\lambda^2\sum_{p=1}^n\sum_{q=1}^n\sum_{r=1}^nd_{pi}d_{qj}d_{rk}\gamma_{pqr}^s,\nonumber
\end{eqnarray}
 for $i,j,k,s=1,2,\dots,n$.

\begin{proposition}
The description of the Rota-Baxter operator of every 2-dimensional Ternary Leibniz algebra is given below.
\begin{itemize}
    \item[$\mathcal{L}_1$ : ]
    $\left(\begin{array}{ccc}
r_{11}&0\\
r_{11}^3-\lambda^2r_{22}-\lambda r_{11}r_{22}-3r_{11}^2r_{22}&r_{22}
\end{array}
\right),\quad 
\left(\begin{array}{ccc}
\frac{i(i\lambda+\sqrt{2}\lambda)}{3}&0\\
r_{21}&\frac{(\lambda+2i\sqrt{2}\lambda)}{9}
\end{array}
\right),
\quad 
\left(\begin{array}{ccc}
0&0\\
r_{21}&0
\end{array}
\right)
;
$

\item[$\mathcal{L}_2$ : ]
    $\left(\begin{array}{ccc}
\frac{-\lambda^2-\lambda r_{22}+r_{22}^2}{3r_{22}}&\frac{-\lambda^2-\lambda r_{22}-2r_{22}^2}{3r_{22}}\\
0&r_{22}
\end{array}
\right),
\quad
\left(\begin{array}{ccc}
\frac{\lambda^2r_{22}+3\lambda r_{22}^2+r_{22}^3}{\lambda(-\lambda+2r_{22})}&\frac{-\lambda+r_{22}^2}{\lambda}\\
0&r_{22}
\end{array}
\right),\quad
\left(\begin{array}{ccc}
0&r_{12}\\
0&0
\end{array}
\right)
,\quad
\left(\begin{array}{ccc}
0&-r_{22}\\
0&r_{22}
\end{array}
\right).
$
\end{itemize}
\end{proposition}

\begin{proposition}
The description of the Rota-Baxter operator of every 3-dimensional Ternary Leibniz algebra is given below.
\begin{itemize}
\item[$\mathcal{L}_1$ : ]
    $\left(\begin{array}{ccc}
0&r_{12}&r_{13}\\
0&0&r_{23}\\
r_{31}&r_{32}&r_{33}\\
\end{array}
\right),\,\left(\begin{array}{ccc}
r_{11}&0&r_{13}\\
-r_{11}&0&r_{23}\\
r_{31}&r_{32}&r_{33}\\
\end{array}
\right),\,
\left(\begin{array}{ccc}
r_{11}&r_{12}&r_{13}\\
r_{21}&r_{22}&r_{23}\\
r_{31}&r_{32}&0
\end{array}
\right),
\,
\left(\begin{array}{ccc}
0&0&r_{13}\\
0&0&r_{23}\\
r_{31}&r_{32}&r_{33}\\
\end{array}
\right)
;
$

\item[$\mathcal{L}_2$ : ]
    $\left(\begin{array}{ccc}
r_{11}&r_{12}&r_{13}\\
r_{21}&0&r_{23}\\
r_{31}&r_{32}&r_{33}\\
\end{array}
\right),\quad
\left(\begin{array}{ccc}
r_{11}&r_{12}&r_{13}\\
r_{21}&r_{22}&-r_{12}-r_{22}\\
r_{31}&r_{32}&r_{33}\\
\end{array}
\right),\quad
\left(\begin{array}{ccc}
0&r_{12}&-r_{33}\\
r_{21}&r_{22}&r_{23}\\
r_{31}&r_{32}&r_{33}\\
\end{array}
\right),$
$
\left(\begin{array}{ccc}
0&r_{12}&0\\
r_{21}&r_{22}&r_{23}\\
0&r_{32}&0\\
\end{array}
\right),\quad
\left(\begin{array}{ccc}
0&r_{12}&0\\
r_{21}&0&r_{23}\\
r_{31}&r_{32}&0\\
\end{array}
\right)
;
$

\item[$\mathcal{L}_3$ : ]
    $\left(\begin{array}{ccc}
r_{11}&r_{12}&-r_{11}\\
0&0&0\\
r_{31}&r_{23}&r_{33}
\end{array}
\right)
,
\quad
\left(\begin{array}{ccc}
r_{11}&r_{12}&-r_{11}\\
r_{33}-r_{11}&r_{22}&r_{11}-r_{33}\\
-r_{33}&-r_{12}-r_{22}&r_{33}
\end{array}
\right),
\,
\left(\begin{array}{ccc}
-r_{31}&r_{12}&r_{13}\\
0&r_{22}&0\\
r_{31}&-r_{12}-r_{22}&r_{33}
\end{array}
\right)
$

$
\left(\begin{array}{ccc}
r_{11}&r_{12}&-r_{11}\\
0&r_{22}&0\\
r_{31}&-r_{22}-r_{22}&-r_{31}
\end{array}
\right),
\quad
\left(\begin{array}{ccc}
-r_{31}&r_{12}&r_{13}\\
0&r_{22}&0\\
r_{31}&-2r_{22}&-r_{31}
\end{array}
\right)
;
$

\item[$\mathcal{L}_4$ : ]
    $\left(\begin{array}{ccc}
r_{11}&0&0\\
r_{21}&0&r_{11}-r_{21}\\
\alpha r_{11}&0&0
\end{array}
\right)
,
\,
\left(\begin{array}{ccc}
r_{11}&0&r_{13}\\
0&0&0\\
r_{31}&0&r_{33}
\end{array}
\right)
,
\,
\left(\begin{array}{ccc}
0&0&0\\
r_{21}&0&r_{31}-r_{21}\\
r_{31}&0&r_{33}
\end{array}
\right)
,
\,
\left(\begin{array}{ccc}
r_{11}&0&0\\
r_{21}&0&r_{11}-r_{21}\\
0&0&r_{11}
\end{array}
\right)
,
\,
\left(\begin{array}{ccc}
0&0&r_{13}\\
r_{21}&0&-r_{13}-r_{21}\\
0&0&-r_{13}
\end{array}
\right)$
;

\item[$\mathcal{L}_5$ : ]
    $\left(\begin{array}{ccc}
r_{11}&-r_{11}&r_{33}\\
r_{11}&-r_{11}&r_{33}\\
0&0&r_{33}
\end{array}
\right),\,
\left(\begin{array}{ccc}
r_{11}&-r_{11}&r_{13}\\
-r_{11}&r_{11}&-r_{13}-r_{33}\\
0&0&r_{33}\\
\end{array}
\right)
,\,
\left(\begin{array}{ccc}
r_{11}&r_{12}&r_{13}\\
r_{21}&r_{22}&r_{23}\\
0&0&0\\
\end{array}
\right),\,
\left(\begin{array}{ccc}
r_{11}&-r_{22}&-2(r_{11}-r_{22})\\
-r_{11}&r_{22}&2(r_{11}-r_{22})\\
0&0&0
\end{array}
\right)
;
$

\item[$\mathcal{L}_6$ : ]
    $\left(\begin{array}{ccc}
0&0&r_{13}\\
0&0&0\\
0&0&0\\
\end{array}
\right),\,\left(\begin{array}{ccc}
0&0&0\\
0&0&-r_{33}\\
0&0&r_{33}\\
\end{array}
\right),\,\left(\begin{array}{ccc}
r_{11}&0&0\\
r_{21}&0&r_{11}-r_{21}\\
0&0&0\\
\end{array}
\right),
\quad
\left(\begin{array}{ccc}
0&0&0\\
0&(\frac{-1}{2}+\frac{i}{2})r_{32}&(\frac{-1}{2}+\frac{i}{2})r_{33}\\
0&r_{32}&r_{33}
\end{array}
\right)
;
$

\item[$\mathcal{L}_7$ : ]
    $\left(\begin{array}{ccc}
r_{11}&0&r_{13}\\
0&0&0\\
r_{31}&0&r_{33}
\end{array}
\right),
\,
\left(\begin{array}{ccc}
r_{11}&0&r_{13}\\
r_{21}&0&r_{23}\\
-r_{11}&0&-r_{13}
\end{array}
\right),
\,
\left(\begin{array}{ccc}
r_{11}&r_{12}&r_{13}\\
0&\lambda&0\\
-r_{11}&-r_{12}&-r_{13}
\end{array}
\right),
\,
\left(\begin{array}{ccc}
r_{11}&r_{12}&r_{13}\\
-4\lambda&9\lambda&-4\lambda\\
\label{}-r_{11}&-r_{12}&\lambda-r_{13}
\end{array}
\right)
;
$

\item[$\mathcal{L}_8$ : ]
    $\left(\begin{array}{ccc}
r_{33}&0&0\\
r_{21}&0&r_{33}-r_{21}\\
0&0&r_{33}
\end{array}
\right),\, \left(\begin{array}{ccc}
r_{11}&0&0\\
r_{21}&0&r_{11}-r_{21}\\
0&0&0
\end{array}
\right),\, \left(\begin{array}{ccc}
-\alpha r_{31}&0&0\\
r_{21}&0&-r_{21}-\alpha r_{31}\\
r_{31}&0&0
\end{array}
\right),\, \left(\begin{array}{ccc}
0&0&0\\
r_{21}&r_{21}&r_{23}\\
0&0&0
\end{array}
\right)
,\,\left(\begin{array}{ccc}
r_{11}&0&r_{13}\\
r_{21}&0&r_{23}\\
0&0&0
\end{array}
\right).$
\end{itemize}
\end{proposition}
Let $(L, \llbracket-, -, -\rrbracket_1, \llbracket-, -, -\rrbracket_2)$ be an $n$-dimensional compatible Ternary Leibniz algebra, $\{e_i\}$ be 
a basis of $L$ and $d$ a derivation on $L$. For any $i, j\in \mathbb{N}, 1\leq i, j, k\leq n$, let us put 
$$\llbracket e_i, e_j, e_k\rrbracket_1=\sum_{p=1}^{n}\chi_{ijk}^pe_p, 
\quad
\llbracket e_i, e_j, e_k\rrbracket_2=\sum_{p=1}^{n}\gamma_{ijk}^pe_p, 
\quad 
 R(e_i) =r_{1i}e_1 +r_{2i}e_2 + \cdots + r_{ni}e_n.
$$
The axioms in [$4)$, Definition \ref{rbc}] is equivalent to
\begin{eqnarray}
   \sum_{r=1}^n \sum_{p=1}^n\sum_{q=1}^nr_{pi}r_{qj}r_{rr}\chi_{pqr}^s
&=&\sum_{r=1}^n \sum_{p=1}^n\sum_{q=1}^n\Bigg(r_{pi}r_{qj}\chi_{pqk}^rr_{sr}+r_{pi}r_{qk}\chi_{pjq}^rr_{rs}
+r_{pj}r_{qk}\chi_{ipq}^rr_{rs}\Bigg)\nonumber\\
&&+\lambda\sum_{r=1}^n\sum_{q=1}^n\Bigg(r_{qi}\chi_{qjk}^rr_{sr}+r_{qj}\chi_{iqk}^rr_{sr}
+r_{qk}\chi_{ijq}^rr_{sr}\Bigg)+\lambda^2  \sum_{r=1}^n\chi_{ijk}^rr_{sr},
\end{eqnarray}
\begin{eqnarray}
   \sum_{r=1}^n \sum_{p=1}^n\sum_{q=1}^nr_{pi}r_{qj}r_{rr}\gamma_{pqr}^s&=&
\sum_{r=1}^n \sum_{p=1}^n\sum_{q=1}^n\Bigg(r_{pi}r_{qj}\gamma_{pqk}^rr_{sr}+r_{pi}r_{qk}\gamma_{pjq}^rr_{rs}
+r_{pj}r_{qk}\gamma_{ipq}^rr_{rs}\Bigg)\nonumber\\
&&+\lambda\sum_{r=1}^n\sum_{q=1}^n\Bigg(r_{qi}\gamma_{qjk}^rr_{sr}+r_{qj}\gamma_{iqk}^rr_{sr}+r_{qk}\gamma_{ijq}^rr_{sr}\Bigg)
+\lambda^2  \sum_{r=1}^n\gamma_{ijk}^rr_{sr},
\end{eqnarray}
for $i,j,k,s=1,2,\dots,n$.
\begin{proposition}
The description of the Rota-Baxter operator of every 2-dimensional Compatible Ternary Leibniz algebra is given below.
\begin{itemize}
    \item[$\mathcal{L}_1$ : ]
    $
\left(\begin{array}{ccc}
0&0\\
0&r_{22}
\end{array}
\right)
;
\qquad\mathcal{L}_2 :
    \left(\begin{array}{ccc}
r_{11}&0\\
\frac{-\lambda^2-\lambda r_{11}-2r^2_{11}}{\lambda+r_{11}}&r_{11}
\end{array}
\right)
,
\quad
\left(\begin{array}{ccc}
0&0\\
r_{21}&0
\end{array}
\right)
;
\qquad\mathcal{L}_3 : 
\left(\begin{array}{ccc}
0&r_{12}\\
0&0
\end{array}
\right)
.$
\end{itemize}
\end{proposition}

\begin{proposition}
The description of the Rota-Baxter operator of every 3-dimensional Compatible Ternary Leibniz algebra is given below.
\begin{itemize}
    \item[$\mathcal{L}_1$ : ]
    $
\left(\begin{array}{ccc}
0&0&0\\
r_{21}&0&0\\
r_{31}&r_{32}&0
\end{array}
\right)
;
\qquad\mathcal{L}_2 :
\left(\begin{array}{ccc}
r_{11}&r_{12}&0\\
0&0&0\\
r_{31}&r_{32}&0\\
\end{array}
\right),\quad
\left(\begin{array}{ccc}
0&r_{12}&0\\
0&0&0\\
r_{31}&r_{32}&0\\
\end{array}
\right)
;
\qquad\mathcal{L}_3 : 
\left(\begin{array}{ccc}
0&0&0\\
r_{21}&0&0\\
0&0&0\\
\end{array}
\right)
;
$
 \item[$\mathcal{L}_4$ : ]
$\left(\begin{array}{ccc}
0&0&r_{13}\\
r_{21}&0&r_{23}\\
0&0&0\\
\end{array}
\right),\quad
\left(\begin{array}{ccc}
0&0&0\\
r_{21}&0&r_{23}\\
0&0&0\\
\end{array}
\right)
;
\qquad\mathcal{L}_5 :
\left(\begin{array}{ccc}
r_{11}&0&0\\
0&0&0\\
0&0&0\\
\end{array}
\right)
;
$
 \item[$\mathcal{L}_7$ : ]
$\left(\begin{array}{ccc}
r_{11}&0&r_{13}\\
0&0&0\\
r_{31}&0&r_{33}\\
\end{array}
\right),\quad
\left(\begin{array}{ccc}
0&0&0\\
r_{21}&0&r_{23}\\
0&0&0
\end{array}
\right)
,\quad
\left(\begin{array}{ccc}
0&0&0\\
-r_{23}&0&r_{23}\\
0&0&0
\end{array}
\right),\quad
\left(\begin{array}{ccc}
r_{11}&0&r_{13}\\
0&0&0\\
-r_{13}&0&r_{33}\\
\end{array}
\right)
;
$

 \item[$\mathcal{L}_7$ : ]
    $
\left(\begin{array}{ccc}
0&0&r_{12}\\
0&0&r_{23}\\
0&0&0\\
\end{array}
\right),\quad
\left(\begin{array}{ccc}
0&0&\lambda -r_{12}\\
0&0&r_{23}\\
0&0&0\\
\end{array}
\right)
;
\qquad\mathcal{L}_8 : 
\left(\begin{array}{ccc}
0&0&0\\
r_{21}&0&r_{23}\\
0&0&0\\
\end{array}
\right)
;
\quad\mathcal{L}_9 :
\left(\begin{array}{ccc}
0&0&0\\
r_{21}&0&r_{23}\\
0&0&0\\
\end{array}
\right),\quad
\left(\begin{array}{ccc}
r_{11}&0&0\\
0&0&0\\
0&0&0\\
\end{array}
\right)
.
$
\end{itemize}
\end{proposition}

\subsection{Nejenhuis operators}

Let $(\mathcal{L}, \llbracket -, -, -\rrbracket)$ be an $n$-dimensional ternary Leibniz algebra with basis $\{e_i\} ( 1\leq i \leq n)$ and 
let $N$ be a Nejenhuis operator on $\mathcal{L}$.
 For any $i, j, k, p\in \mathbb{N}, 1\leq i, j, k, p\leq n$, let us put 
$$\llbracket e_i, e_j, e_k\rrbracket=\sum_{p=1}^{n}\gamma_{ijk}^pe_p \quad\mbox{and}\quad
  N(e_i) =N_{1i}e_1 +N_{2i}e_2 + \cdots + N_{ni}e_n,$$
Then, in term of basis elements, equation (\ref{njid}) is equivalent to :
\begin{eqnarray}
 \sum_{P=1}^n \sum_{q=1}^n\sum_{r=1}^nN_{pi}N_{qj}N_{rj}\gamma_{pqr}^s
&=&\sum_{r=1}^n \sum_{p=1}^n\sum_{q=1}^n\Bigg(N_{pi}N_{qj}\gamma_{pqk}^rN_{sr}+N_{pi}N_{qk}\gamma_{pjq}^rN_{rs}+N_{pj}N_{qk}\gamma_{ipq}^rN_{sr}
\Bigg)\nonumber\\
&&-\sum_{P=1}^n \sum_{q=1}^n\sum_{R=1}^n\Big(N_{pi}\gamma_{pjk}^rN_{rq}N_{sr}
+N_{pj}\gamma_{ipk}^qN_{rq}N_{sr}+N_{pk}\gamma_{ijk}^qN_{rq}N_{sr}\Big)\nonumber\\
&&+\sum_{P=1}^n \sum_{q=1}^n\sum_{r=1}^n\gamma_{ijk}^pN_{qp}N_{rq}N_{sr},\label{rot}
\end{eqnarray}
 for $i,j,k,s=1,2,\dots,n$.

\begin{proposition}
The description of the  Nejenhuis operator of every 2-dimensional Ternary Leibniz algebra is given below.
\begin{itemize}
    \item[$\mathcal{L}_1$ : ]
    $\left(\begin{array}{ccc}
N_{11}&0\\
N_{21}&N_{11}
\end{array}
\right);
\qquad\mathcal{L}_2 : 
    \left(\begin{array}{ccc}
N_{11}&N_{12}\\
0&N_{11}
\end{array}
\right),\,
 \left(\begin{array}{ccc}
N_{11}&N_{11}-N_{22}\\
0&N_{22}
\end{array}
\right)
.
$
\end{itemize}
\end{proposition}

\begin{proposition}
The description of the  Nejenhuis operator of every 3-dimensional Ternary Leibniz algebra is given below.
\begin{itemize}
    \item[$\mathcal{L}_1$ : ]
    $\left(\begin{array}{ccc}
N_{33}&0&0\\
0&N_{33}&0\\
N_{31}&N_{32}&N_{33}
\end{array}
\right),
\quad
\left(\begin{array}{ccc}
N_{33}&0&0\\
N_{21}&N_{33}&0\\
N_{31}&N_{32}&N_{33}
\end{array}
\right),
\quad
\left(\begin{array}{ccc}
N_{33}&0&0\\
N_{21}&N_{22}&0\\
N_{31}&N_{32}&N_{33}
\end{array}
\right);
$
 
\item[$\mathcal{L}_2$ : ]
  $\left(\begin{array}{ccc}
N_{22}&0&0\\
N_{21}&N_{22}&N_{23}\\
0&0&N_{22}
\end{array}
\right),\quad
\left(\begin{array}{ccc}
\frac{(4N_{22}-N_{31}-iN_{31}\sqrt{7})}{4}&0&0\\
0&N_{22}&0\\
N_{31}&0&N_{22}
\end{array}
\right),
\quad
\left(\begin{array}{ccc}
N_{22}&0&N_{13}\\
0&N_{22}&\\
0&0&N_{22}
\end{array}
\right)
;
$

\item[$\mathcal{L}_3$ : ]
  $\left(\begin{array}{ccc}
N_{22}&N_{12}&0\\
N_{31}&N_{32}&N_{22}\\
0&0&N_{22}
\end{array}
\right),
\quad
\left(\begin{array}{ccc}
iN_{13}+N_{22}&N_{12}&N_{13}\\
0&N_{22}&\\
N_{13}&n_{32}&-i(N_{13}+iN_{22})
\end{array}
\right)
;
$

\item[$\mathcal{L}_4$ : ]
  $
\left(\begin{array}{ccc}
N_{11}&0&0\\
0&N_{22}&\\
N_{31}&0&N_{33}
\end{array}
\right)
;
$

\item[$\mathcal{L}_5$ : ]
  $\left(\begin{array}{ccc}
N_{11}&0&N_{13}\\
0&N_{33}&N_{23}\\
0&0&N_{33}
\end{array}
\right),
\quad
\left(\begin{array}{ccc}
N_{11}+N_{33}&N_{12}&N_{13}\\
0&N_{33}&N_{12}-N_{13}\\
0&0&N_{33}
\end{array}
\right)
,
\quad
\left(\begin{array}{ccc}
N_{33}&N_{12}&N_{13}\\
0&N_{33}&N_{23}\\
0&0&N_{33}
\end{array}
\right)
;
$

\item[$\mathcal{L}_6$ : ]
  $\left(\begin{array}{ccc}
N_{11}&-N_{32}&0\\
0&N_{11}-N_{32}&0\\
0&N_{32}&N_{11}
\end{array}
\right),
\quad
\left(\begin{array}{ccc}
N_{11}&N_{12}&N_{13}\\
0&N_{11}&\\
0&0&N_{11}
\end{array}
\right)
,
\quad
\left(\begin{array}{ccc}
N_{11}&N_{32}&-N_{23}\\
0&N_{11}-N_{32}&N_{23}\\
0&N_{32}&N_{11}-N_{23}
\end{array}
\right)
;
$

\item[$\mathcal{L}_7$ : ]
  $\left(\begin{array}{ccc}
N_{11}&0&N_{13}\\
0&N_{22}&0\\
N_{31}&0&N_{33}
\end{array}
\right),
\quad
\left(\begin{array}{ccc}
N_{11}&-N_{32}&0\\
0&N_{22}&0\\
0&N_{32}&N_{11}
\end{array}
\right)
,
\quad
\left(\begin{array}{ccc}
N_{11}&-N_{32}&-N_{31}\\
0&N_{22}&0\\
N_{31}&N_{32}&N_{11}+2N_{31}
\end{array}
\right)
;
$

\item[$\mathcal{L}_8$ : ]
  $\left(\begin{array}{ccc}
N_{11}&0&N_{13}\\
N_{21}&N_{22}&N_{23}\\
0&0&N_{22}
\end{array}
\right),
\quad
\left(\begin{array}{ccc}
N_{22}&0&N_{13}\\
N_{21}&N_{22}&N_{23}\\
0&0&N_{22}
\end{array}
\right),\quad
\left(\begin{array}{ccc}
N_{22}&0&0\\
N_{21}&N_{22}&N_{23}\\
0&0&N_{22}
\end{array}
\right).
$
\end{itemize}
\end{proposition}

Let $(L, \llbracket-, -, -\rrbracket_1, \llbracket-, -, -\rrbracket_2)$ be an $n$-dimensional compatible Ternary Leibniz algebra, $\{e_i\}$ be 
a basis of $L$ and $N$ a Nijenhuis operator on $L$. For any $i, j\in \mathbb{N}, 1\leq i, j, k\leq n$, let us put 
$$\llbracket e_i, e_j, e_k\rrbracket_1=\sum_{p=1}^{n}\chi_{ijk}^pe_p, 
\quad
\llbracket e_i, e_j, e_k\rrbracket_2=\sum_{p=1}^{n}\gamma_{ijk}^pe_p, 
\quad 
N(e_i) =N_{1i}e_1 +N_{2i}e_2 + \cdots + N_{ni}e_n.
$$
The axioms in [$4)$, Definition \ref{njc}] are equivalent to

\begin{eqnarray}
\sum_{P=1}^n \sum_{q=1}^n\sum_{r=1}^nN_{pi}N_{qj}N_{rj}\chi_{pqr}^s
&=&\sum_{r=1}^n \sum_{p=1}^n\sum_{q=1}^n\Bigg(N_{pi}N_{qj}\chi_{pqk}^rN_{sr}
+N_{pi}N_{qk}\chi_{pjq}^rN_{rs}+N_{pj}N_{qk}\chi_{ipq}^rN_{sr}\Bigg)\nonumber\\
&&-\sum_{P=1}^n \sum_{q=1}^n\sum_{R=1}^n\Big(N_{pi}\chi_{pjk}^rN_{rq}N_{sr}
+N_{pj}\chi_{ipk}^qN_{rq}N_{sr}+N_{pk}\chi_{ijk}^qN_{rq}N_{sr}\Big)\nonumber\\
&&+\sum_{P=1}^n \sum_{q=1}^n\sum_{r=1}^n\chi_{ijk}^pN_{qp}N_{rq}N_{sr},\nonumber\label{nej}
\end{eqnarray}
and
\begin{eqnarray}
\sum_{P=1}^n \sum_{q=1}^n\sum_{r=1}^nN_{pi}N_{qj}N_{rj}\gamma_{pqr}^s
&=&\sum_{r=1}^n \sum_{p=1}^n\sum_{q=1}^n\Bigg(N_{pi}N_{qj}\gamma_{pqk}^rN_{sr}+N_{pi}N_{qk}\gamma_{pjq}^rN_{rs}
+N_{pj}N_{qk}\gamma_{ipq}^rN_{sr}\Bigg)\nonumber\\
&&-\sum_{P=1}^n \sum_{q=1}^n\sum_{R=1}^n\Big(N_{pi}\gamma_{pjk}^rN_{rq}N_{sr}+N_{pj}\gamma_{ipk}^qN_{rq}N_{sr}+N_{pk}\gamma_{ijk}^qN_{rq}N_{sr}
\Big)\nonumber\\
&&+\sum_{P=1}^n \sum_{q=1}^n\sum_{r=1}^n\gamma_{ijk}^pN_{qp}N_{rq}N_{sr},\label{rot}\nonumber
\end{eqnarray}
for $i,j,k,s=1,2,\dots,n$.

\begin{proposition}
The description of the  Nejenhuis operator of every 2-dimensional compatible Ternary Leibniz algebra is given below.
\begin{itemize}
    \item[$\mathcal{L}_1$ : ]
    $\left(\begin{array}{ccc}
N_{11}&0\\
N_{21}&N_{11}
\end{array}
\right),\quad 
\left(\begin{array}{ccc}
N_{11}&0\\
0&N_{22}
\end{array}
\right);
\qquad\mathcal{L}_2 : 
\left(\begin{array}{ccc}
N_{11}&0\\
N_{21}&N_{11}
\end{array}
\right)
;
\qquad\mathcal{L}_3 : 
    \left(\begin{array}{ccc}
N_{11}&N_{12}\\
0&N_{11}
\end{array}
\right),\quad
\left(\begin{array}{ccc}
N_{11}&0\\
0&N_{11}
\end{array}
\right)
.
$
\end{itemize}
\end{proposition}

\begin{proposition}
The description of the  Nejenhuis operator of every 3-dimensional compatible Ternary Leibniz algebra is given below.
\begin{itemize}
    \item[$\mathcal{L}_1$ : ]
    $\left(\begin{array}{ccc}
N_{11}&0&0\\
N_{21}&N_{11}&N_{23}\\
0&0&N_{11}
\end{array}
\right),\quad 
\left(\begin{array}{ccc}
N_{11}&0&0\\
N_{21}&N_{11}&0\\
N_{31}&N_{32}&N_{11}
\end{array}
\right),\quad 
\left(\begin{array}{ccc}
N_{11}&0&0\\
N_{21}&N_{11}+N_{23}&N_{23}\\
N_{23}-N_{21}&0&N_{11}
\end{array}
\right);
$

\item[$\mathcal{L}_2$ : ]
    $\left(\begin{array}{ccc}
N_{11}&N_{12}&0\\
0&N_{33}&0\\
N_{31}&N_{32}&N_{33}
\end{array}
\right),\quad 
\left(\begin{array}{ccc}
N_{33}&0&0\\
0&N_{33}&0\\
N_{31}&N_{32}&N_{33}
\end{array}
\right),\quad 
\left(\begin{array}{ccc}
N_{11}&0&0\\
0&N_{33}&0\\
N_{31}&N_{32}&N_{33}
\end{array}
\right);
$

\item[$\mathcal{L}_3$ : ]
    $\left(\begin{array}{ccc}
N_{33}&0&N_{13}\\
N_{21}&N_{33}&N_{23}\\
0&0&N_{33}
\end{array}
\right),\quad 
\left(\begin{array}{ccc}
N_{33}&0&N_{13}\\
0&N_{33}&N_{23}\\
0&0&N_{33}
\end{array}
\right),\quad 
\left(\begin{array}{ccc}
2N_{33}-N_{22}&N_{12}&N_{13}\\
-\frac{(N_{22}-N_{33})^2}{N_{12}}&N_{22}&N_{23}\\
0&0&N_{33}
\end{array}
\right);
$

\item[$\mathcal{L}_4$ : ]
    $\left(\begin{array}{ccc}
N_{22}&0&N_{13}\\
N_{21}&N_{22}&N_{23}\\
0&0&N_{22}
\end{array}
\right),\quad 
\left(\begin{array}{ccc}
N_{22}&0&N_{13}\\
N_{21}&N_{22}&0\\
0&0&N_{22}
\end{array}
\right),\quad 
\left(\begin{array}{ccc}
N_{22}&0&0\\
N_{21}&N_{22}&N_{23}\\
0&0&N_{22}
\end{array}
\right);
$

\item[$\mathcal{L}_5$ : ]
    $\left(\begin{array}{ccc}
N_{22}&N_{12}&0\\
0&N_{22}&0\\
N_{31}&N_{32}&N_{22}
\end{array}
\right),\quad 
\left(\begin{array}{ccc}
N_{22}&N_{12}&0\\
0&N_{22}&0\\
0&N_{32}&N_{33}
\end{array}
\right),\quad 
\left(\begin{array}{ccc}
N_{22}-N_{12}&N_{12}&0\\
0&N_{22}&0\\
0&N_{32}&N_{22}
\end{array}
\right);
$

\item[$\mathcal{L}_6$ : ]
    $\left(\begin{array}{ccc}
N_{33}&0&0\\
N_{21}&N_{33}&N_{23}\\
0&0&N_{33}
\end{array}
\right),\quad
\left(\begin{array}{ccc}
N_{11}&0&0\\
0&N_{22}&0\\
0&0&N_{33}
\end{array}
\right),\quad
\left(\begin{array}{ccc}
N_{33}&0&N_{13}\\
0&N_{22}&0\\
0&0&N_{33}
\end{array}
\right),\,
\left(\begin{array}{ccc}
N_{11}&0&N_{13}\\
0&N_{22}&0\\
N_{31}&0&N_{33}
\end{array}
\right);
$

\item[$\mathcal{L}_7$ : ]
    $ 
\left(\begin{array}{ccc}
N_{33}&0&N_{13}\\
N_{21}&N_{33}&N_{23}\\
0&0&N_{33}
\end{array}
\right),\quad
\left(\begin{array}{ccc}
N_{33}&0&N_{13}\\
0&N_{33}&N_{23}\\
0&0&N_{33}
\end{array}
\right);
$

\item[$\mathcal{L}_8$ : ]
    $\left(\begin{array}{ccc}
N_{22}&0&0\\
N_{21}&N_{22}&N_{23}\\
0&0&N_{22}
\end{array}
\right);
$

\item[$\mathcal{L}_9$ : ]
    $\left(\begin{array}{ccc}
N_{22}&N_{12}&0\\
0&N_{22}&0\\
N_{31}&N_{32}&N_{22}
\end{array}
\right),\qquad 
\left(\begin{array}{ccc}
N_{22}&N_{12}&0\\
0&N_{22}&0\\
0&N_{32}&N_{22}
\end{array}
\right).
$
\end{itemize}
\end{proposition}

\subsection{Reynolds operators}
Let $(\mathcal{L}, \llbracket -, -, -\rrbracket)$ be an $n$-dimensional ternary Leibniz algebra with basis $\{e_i\} ( 1\leq i \leq n)$ and 
let $R$ be a Reynold operator on $\mathcal{L}$.
 For any $i, j, k, p\in \mathbb{N},
 1\leq i, j, k, p\leq n$, let us put 
$$\llbracket e_i, e_j, e_k\rrbracket=\sum_{p=1}^{n}\gamma_{ijk}^pe_p \quad\mbox{and}\quad R(e_i) = R_{1i}e_1 + R_{2i}e_2 + \cdots + R_{ni}e_n,$$

Then, in term of basis elements, equation (\ref{reti}) is equivalent to :
\begin{eqnarray}
\sum_{s=1}^n\sum_{r=1}^n\sum_{q=1}^nR_{qi}R_{rj}R_{sk}\gamma_{qrs}^t
&=&\sum_{s=1}^n\sum_{r=1}^n\sum_{q=1}^n\Bigg(R_{qi}R_{rj}R_{sk}\gamma_{qrs}^t
+R_{qi}R_{rj}\gamma_{qrk}^sR_{ts}+R_{qi}R_{rk}\gamma_{qjr}^sR_{ts}+R_{qj}R_{rk}\gamma_{iqr}^sR_{ts}\Bigg)\nonumber\\
&&-\sum_{s=1}^n\sum_{P=1}^n \sum_{q=1}^n\sum_{r=1}^nR_{pi}R_{qj}R_{rk}\gamma_{pqr}^sR_{ts}\nonumber\label{rot},
\end{eqnarray}
for $i,j,k,s=1,2,\dots,n$.
\begin{proposition}
The description of the Reynolds operator of every 2-dimensional Ternary Leibniz algebra is given below.
\begin{itemize}
    \item[$\mathcal{L}_1$ : ]
$\left(\begin{array}{ccc}
R_{11}&0\\
R_{21}&\frac{R_{11}}{3-R_{11}}
\end{array}
\right),\,
\left(\begin{array}{ccc}
0&0\\
R_{21}&R_{22}
\end{array}
\right)
;
\qquad\mathcal{L}_2 :
\left(\begin{array}{ccc}
R_{11}&R_{12}\\
0&2
\end{array}
\right),\,
\left(\begin{array}{ccc}
0&R_{12}\\
0&2
\end{array}
\right)
,\,
\left(\begin{array}{ccc}
0&-R_{22}\\
0&R_{22}
\end{array}
\right)
.
$
\end{itemize}
\end{proposition}

\begin{proposition}
The description of the Reynolds operator of every 3-dimensional Ternary Leibniz algebra is given below.
\begin{itemize}
    \item[$\mathcal{L}_1$ : ]
    $
 \left(\begin{array}{ccc}
0&0&0\\
0&0&0\\
R_{31}&R_{32}&R_{33}
\end{array}
\right),\quad \left(\begin{array}{ccc}
\frac{3R_{33}}{1+R_{33}}&0&0\\
0&\frac{3R_{33}}{1+R_{33}}&0\\
R_{31}&R_{32}&R_{33}
\end{array}
\right);
$

\item[$\mathcal{L}_2$ : ]
    $\left(\begin{array}{ccc}
0&0&0\\
R_{21}&R_{22}&R_{23}\\
0&0&0\\
\end{array}
\right),
\quad
\left(\begin{array}{ccc}
0&0&R_{13}\\
R_{21}&0&R_{23}\\
0&0&0\\
\end{array}
\right),
\quad
 \left(\begin{array}{ccc}
\frac{3R_{22}}{1-R_{22}}&0&0\\
R_{21}&R_{22}&R_{23}\\
0&0&\frac{3R_{22}}{1-R_{22}}
\end{array}
\right);
$

\item[$\mathcal{L}_3$ : ]
    $\left(\begin{array}{ccc}
    R_{11}&R_{12}&-R_{11}\\
0&R_{22}&0\\
-R_{11}&-R_{12}-R_{22}&R_{11}
\end{array}
\right),
\quad
\left(\begin{array}{ccc}
R_{11}&R_{12}&R_{11}\\
-R_{11}-R_{31}&0&R_{11}+R_{31}\\
R_{31}&-R_{12}&-R_{31}
\end{array}
\right);
$
\item[$\mathcal{L}_4$ : ]
    $\left(\begin{array}{ccc}
    R_{11}&R_{12}&R_{13}\\
0&R_{22}&0\\
-R_{11}&-R_{12}&-R_{13}
\end{array}
\right),
\quad
\left(\begin{array}{ccc}
R_{11}&0&R_{13}\\
0&R_{22}&0\\
-R_{11}&0&-R_{31}
\end{array}
\right),
\quad
\left(\begin{array}{ccc}
R_{11}&R_{12}&R_{13}\\
R_{21}&R_{22}&R_{23}\\
-R_{11}&-R_{12}&-R_{31}
\end{array}
\right);
$

\item[$\mathcal{L}_5$ : ]
    $\left(\begin{array}{ccc}
R_{11}&R_{12}&R_{13}\\
R_{21}&R_{22}&R_{23}\\
0&0&0\\
\end{array}
\right),
\quad \left(\begin{array}{ccc}
R_{11}&-R_{22}&R_{13}\\
-R_{11}&R_{22}&-R_{13}\\
0&0&0
\end{array}
\right),\quad
 \left(\begin{array}{ccc}
R_{11}&-R_{11}&R_{13}\\
-R_{11}&R_{11}&-R_{13}-R_{33}\\
0&0&R_{33}
\end{array}
\right)
$

\item[$\mathcal{L}_6$ : ]
$
\left(\begin{array}{ccc}
R_{11}&R_{12}&R_{13}\\
0&\frac{3R_{11}}{1+R_{11}}&0\\
0&0&\frac{3R_{11}}{1+R_{11}}
\end{array}
\right);
$
\item[$\mathcal{L}_7$ : ]
$\left(\begin{array}{ccc}
R_{11}&0&R_{13}\\
0&R_{22}&0\\
-R_{11}&0&-R_{13}
\end{array}
\right),
\quad
\left(\begin{array}{ccc}
R_{11}&R_{12}&R_{13}\\
0&R_{22}&0\\
-R_{11}&-R_{12}&-R_{13}
\end{array}
\right),
\quad
\left(\begin{array}{ccc}
R_{11}&R_{12}&R_{13}\\
R_{21}&R_{22}&R_{23}\\
-R_{11}&0&-R_{13}
\end{array}
\right)
;
$
\item[$\mathcal{L}_8$ : ]
$\left(\begin{array}{ccc}
\frac{3R_{22}}{1+R_{22}}&0&0\\
R_{21}&R_{22}&R_{23}\\
0&0&\frac{3R_{22}}{1+R_{22}}
\end{array}
\right),
\quad
\left(\begin{array}{ccc}
R_{11}&0&R_{13}\\
R_{21}&0&R_{23}\\
0&0&0
\end{array}
\right),
\quad
\left(\begin{array}{ccc}
0&0&R_{13}\\
R_{21}&0&R_{23}\\
0&0&0
\end{array}
\right),
\quad
\left(\begin{array}{ccc}
0&0&0\\
R_{21}&R_{22}&R_{23}\\
0&0&0
\end{array}
\right).
$
\end{itemize}
\end{proposition}
Let $(L, \llbracket-, -, -\rrbracket_1, \llbracket-, -, -\rrbracket_2)$ be an $n$-dimensional compatible Ternary Leibniz algebra, $\{e_i\}$ be 
a basis of $L$ and $R$ be a Reynolds operator on $L$. For any $i, j\in \mathbb{N}, 1\leq i, j, k\leq n$, let us put 
$$\llbracket e_i, e_j, e_k\rrbracket_1=\sum_{p=1}^{n}\chi_{ijk}^pe_p, 
\quad
\llbracket e_i, e_j, e_k\rrbracket_2=\sum_{p=1}^{n}\gamma_{ijk}^pe_p, 
\quad 
R(e_i) = R_{1i}e_1 + R_{2i}e_2 + \cdots + R_{ni}e_n.
$$
The axioms in Definition \ref{ret} are respectively equivalent to

%
%
\begin{eqnarray}
 \sum_{s=1}^n\sum_{r=1}^n\sum_{q=1}^nR_{qi}R_{rj}R_{sk}\chi_{qrs}^t
&=&\sum_{s=1}^n\sum_{r=1}^n\sum_{q=1}^n\Bigg(R_{qi}R_{rj}R_{sk}\chi_{qrs}^t+R_{qi}R_{rj}\chi_{qrk}^sR_{ts}
+R_{qi}R_{rk}\chi_{qjr}^sR_{ts}+R_{qj}R_{rk}\chi_{iqr}^sR_{ts}\Bigg)\nonumber\\
&&-\sum_{s=1}^n\sum_{P=1}^n \sum_{q=1}^n\sum_{r=1}^nR_{pi}R_{qj}R_{rk}\chi_{pqr}^sR_{ts}\nonumber\label{rot},
\end{eqnarray}
and
\begin{eqnarray}
 \sum_{s=1}^n\sum_{r=1}^n\sum_{q=1}^nR_{qi}R_{rj}R_{sk}\gamma_{qrs}^t
&=&\sum_{s=1}^n\sum_{r=1}^n\sum_{q=1}^n\Bigg(R_{qi}R_{rj}R_{sk}\gamma_{qrs}^t+R_{qi}R_{rj}\gamma_{qrk}^sR_{ts}+R_{qi}R_{rk}\gamma_{qjr}^sR_{ts}
+R_{qj}R_{rk}\gamma_{iqr}^sR_{ts}\Bigg)\nonumber\\
&&-\sum_{s=1}^n\sum_{P=1}^n \sum_{q=1}^n\sum_{r=1}^nR_{pi}R_{qj}R_{rk}\gamma_{pqr}^sR_{ts}\nonumber\label{rot},\nonumber
\end{eqnarray}
for $i,j,k,s=1,2,\dots,n$.

\begin{proposition}
The description of the Reynolds operator of every 2-dimensional Ternary Leibniz algebra is given below.
\begin{itemize}
    \item[$\mathcal{L}_1$ : ]
$\left(\begin{array}{ccc}
0&0\\
R_{21}&R_{22}
\end{array}
\right);
\quad\mathcal{L}_2 : 
\left(\begin{array}{ccc}
0&0\\
R_{21}&R_{22}
\end{array}
\right),\quad
\left(\begin{array}{ccc}
R_{11}&0\\
R_{21}&\frac{-R_{11}}{R_{11}+3}
\end{array}
\right)
;
\quad\mathcal{L}_3 :
\left(\begin{array}{ccc}
R_{11}&R_{12}\\
0&0
\end{array}
\right)
,\quad
\left(\begin{array}{ccc}
0&R_{12}\\
0&0
\end{array}
\right)
.
$
\end{itemize}
\end{proposition}

\begin{proposition}
The description of the Reynolds operator of every 3-dimensional Compatible Ternary Leibniz algebra is given below.
\begin{itemize}
    \item[$\mathcal{L}_1$ : ]
$\left(\begin{array}{ccc}
0&0&0\\
R_{21}&R_{22}&R_{23}\\
R_{31}&R_{32}&R_{33}\\
\end{array}
\right),\quad
\left(\begin{array}{ccc}
R_{21}&0&0\\
-R_{21}&-R_{32}&-R_{33}\\
0&R_{32}&R_{33}
\end{array}
\right),\quad
\left(\begin{array}{ccc}
0&0&0\\
R_{21}&-R_{32}&0\\
-R_{21}&R_{32}&0
\end{array}
\right);
$

\item[$\mathcal{L}_2$ : ]
$\left(\begin{array}{ccc}
R_{11}&R_{12}&0\\
0&0&0\\
R_{31}&R_{32}&0
\end{array}
\right),\quad
\left(\begin{array}{ccc}
R_{11}&0&0\\
0&0&0\\
R_{31}&R_{32}&0
\end{array}
\right),\quad
\left(\begin{array}{ccc}
0&0&0\\
0&0&0\\
R_{31}&R_{32}&R_{33}
\end{array}
\right),\quad
\left(\begin{array}{ccc}
\frac{3R_{33}}{1+R_{33}}&0&0\\
0&\frac{3R_{33}}{1+R_{33}}&0\\
R_{31}&R_{32}&R_{33}
\end{array}
\right);
$
\item[$\mathcal{L}_3$ : ]
$\left(\begin{array}{ccc}
R_{11}&R_{12}&R_{13}\\
R_{21}&R_{22}&R_{23}\\
0&0&0
\end{array}
\right),\quad
\left(\begin{array}{ccc}
R_{11}&0&R_{13}\\
R_{21}&0&R_{23}\\
0&0&0\\
\end{array}
\right),\quad
\left(\begin{array}{ccc}
R_{11}&0&R_{13}\\
-\alpha R_{11}&0&-\alpha R_{13}\\
0&0&0
\end{array}
\right);
$

 \item[$\mathcal{L}_4$ : ]
$\left(\begin{array}{ccc}
R_{11}&0&R_{13}\\
R_{21}&0&R_{23}\\
0&0&0
\end{array}
\right),\quad
\left(\begin{array}{ccc}
0&0&0\\
R_{21}&R_{22}&R_{23}\\
0&0&0
\end{array}
\right),\quad
\left(\begin{array}{ccc}
R_{11}&0&0\\
R_{21}&0&R_{23}\\
0&0&0
\end{array}
\right);
$

 \item[$\mathcal{L}_5$ : ]
$\left(\begin{array}{ccc}
R_{11}&0&0\\
0&0&0\\
0&0&0
\end{array}
\right),\quad
\left(\begin{array}{ccc}
R_{11}&R_{12}&R_{13}\\
0&0&0\\
R_{31}&R_{32}&R_{33}
\end{array}
\right),\quad
\left(\begin{array}{ccc}
-R_{33}&R_{12}&-R_{33}\\
0&-R_{12}&0\\
R_{33}&0&R_{33}
\end{array}
\right),\quad
\left(\begin{array}{ccc}
0&0&0\\
0&0&0\\
R_{31}&R_{32}&0
\end{array}
\right);
$
 \item[$\mathcal{L}_6$ : ]
$\left(\begin{array}{ccc}
R_{11}&0&R_{13}\\
0&0&0\\
R_{31}&0&R_{33}
\end{array}
\right),\quad\left(\begin{array}{ccc}
0&0&0\\
R_{21}&R_{22}&R_{23}\\
0&0&0
\end{array}
\right),\quad\left(\begin{array}{ccc}
0&0&0\\
R_{21}&0&R_{23}\\
0&0&0
\end{array}
\right),\quad\left(\begin{array}{ccc}
0&0&0\\
0&R_{22}&0\\
0&0&0
\end{array}
\right);
$

 \item[$\mathcal{L}_7$ : ]
$\left(\begin{array}{ccc}
R_{11}&R_{12}&R_{13}\\
R_{21}&R_{22}&R_{23}\\
0&0&0
\end{array}
\right),\quad
\left(\begin{array}{ccc}
0&0&R_{13}\\
0&0&R_{23}\\
0&0&0
\end{array}
\right),\quad
\left(\begin{array}{ccc}
0&-R_{22}&-R_{23}\\
0&R_{22}&R_{23}\\
0&0&0
\end{array}
\right);
$
 \item[$\mathcal{L}_8$ : ]
$\left(\begin{array}{ccc}
0&0&0\\
R_{21}&R_{22}&R_{23}\\
0&0&0
\end{array}
\right),\qquad
\left(\begin{array}{ccc}
0&0&0\\
R_{21}&0&R_{23}\\
0&0&0
\end{array}
\right);
$

\item[$\mathcal{L}_9$ : ]
$\left(\begin{array}{ccc}
R_{11}&R_{12}&R_{13}\\
0&0&0\\
R_{31}&R_{32}&R_{33}\\
\end{array}
\right),\qquad
\left(\begin{array}{ccc}
0&R_{12}&0\\
0&0&0\\
0&R_{32}&0
\end{array}
\right).
$
\end{itemize}
\end{proposition}

\subsection{Elements of centroids}

Let $(\mathcal{L}, \llbracket -, -, -\rrbracket)$ be an $n$-dimensional ternary Leibniz algebra with basis $\{e_i\} ( 1\leq i \leq n)$ and 
let $C$ be an element of centroid on $\mathcal{L}$.
 For any $i, j, k, p\in \mathbb{N}, 1\leq i, j, k, p\leq n$, let us put 
$$\llbracket e_i, e_j, e_k\rrbracket=\sum_{p=1}^{n}\gamma_{ijk}^pe_p \quad\mbox{and}\quad
  C(e_i) = c_{1i}e_1 + c_{2i}e_2 + \cdots + c_{ni}e_n,$$
Then, in term of basis elements, equation (\ref{ec1}) is equivalent to :
\begin{align*}
   \sum_{p=1}^n\gamma_{ijk}^pc_{qp}=\sum_{p=1}^nc_{pi}\gamma_{pjk}^q=\sum_{p=1}^nc_{pj}\gamma_{ipk}^q=\sum_{q=1}^nc_{pk}\gamma_{ijp}^q
\end{align*} for $i,j,,k,q=1,2,\dots,n$.

\begin{proposition}
The description of the centroids of every 2-dimensional Ternary Leibniz algebra is given below.
\begin{itemize}
    \item[$\mathcal{L}_1$ : ]
    $\left(\begin{array}{ccc}
c_{11}&0\\
0&c_{11}
\end{array}
\right);
\qquad
\mathcal{L}_2 : 
    \left(\begin{array}{ccc}
c_{11}&0\\
0&c_{11}
\end{array}
\right)
.
$
\end{itemize}
\end{proposition}

\begin{proposition}
The description of the centroids of every 3-dimensional Ternary Leibniz algebra is given below.
\begin{itemize}
    \item[$\mathcal{L}_1$ : ]
    $\left(\begin{array}{ccc}
c_{11}&0&0\\
0&c_{11}&0\\
0&0&c_{11}
\end{array}
\right);\qquad
\mathcal{L}_2 : 
    \left(\begin{array}{ccc}
c_{11}&0&0\\
0&c_{11}&0\\
0&0&c_{11}
\end{array}
\right)
;
$

\item[$\mathcal{L}_3$ : ]
    $\left(\begin{array}{ccc}
c_{11}&c_{12}&c_{13}\\
0&c_{11}+c_{13}&0\\
c_{31}&c_{12}&c_{11}
\end{array}
\right)
;
\qquad
\mathcal{L}_4 :
\left(\begin{array}{ccc}
c_{11}&0&0\\
0&c_{11}&0\\
0&0&c_{11}
\end{array}
\right)
;
\qquad\mathcal{L}_5 : 
    \left(\begin{array}{ccc}
c_{11}&c_{12}&c_{13}\\
c_{21}&c_{11}&-c_{13}\\
0&0&c_{11}+c_{12}
\end{array}
\right)
;
$

\item[$\mathcal{L}_6$ : ]
    $\left(\begin{array}{ccc}
c_{11}&0&0\\
0&c_{11}&0\\
0&0&c_{11}
\end{array}
\right)
;
\qquad\mathcal{L}_7 : 
    \left(\begin{array}{ccc}
c_{11}&0&c_{13}\\
0&c_{22}&0\\
c_{22}-c_{11}&0&c_{22}-c_{13}
\end{array}
\right)
;
\qquad\mathcal{L}_8: 
    \left(\begin{array}{ccc}
c_{22}&0&0\\
0&c_{22}&0\\
0&0&c_{22}
\end{array}
\right).
$
\end{itemize}
\end{proposition}
Let $(L, \llbracket-, -, -\rrbracket_1, \llbracket-, -, -\rrbracket_2)$ be an $n$-dimensional compatible Ternary Leibniz algebra, $\{e_i\}$ be 
a basis of $L$ and $d$ an element of centroid on $L$. For any $i, j\in \mathbb{N}, 1\leq i, j, k\leq n$, let us put 
$$\llbracket e_i, e_j, e_k\rrbracket_1=\sum_{p=1}^{n}\chi_{ijk}^pe_p, 
\quad
\llbracket e_i, e_j, e_k\rrbracket_2=\sum_{p=1}^{n}\gamma_{ijk}^pe_p, 
\quad 
 C(e_i) = c_{1i}e_1 + c_{2i}e_2 + \cdots + c_{ni}e_n.
$$
The axioms in [$2)$, Definition \ref{ec}] are equivalent to
\begin{align*}
   \sum_{p=1}^n\chi_{ijk}^pc_{qp}=\sum_{p=1}^nc_{pi}\chi_{pjk}^q=\sum_{p=1}^nc_{pj}\chi_{ipk}^q=\sum_{q=1}^nc_{pk}\chi_{ijp}^q,\\
     \sum_{p=1}^n\gamma_{ijk}^pc_{qp}=\sum_{p=1}^nc_{pi}\gamma_{pjk}^q=\sum_{p=1}^nc_{pj}\gamma_{ipk}^q=\sum_{q=1}^nc_{pk}\gamma_{ijp}^q
\end{align*} 
for $i,j,,k,q=1,2,\dots,n$.

\begin{proposition}
The description of the centroids of every 2-dimensional compatible Ternary Leibniz algebra is given below.
\begin{itemize}
    \item 
    $\mathcal{L}_1 :
    \left(\begin{array}{cccc}
c_{11}&0\\
0&c_{11}\\
\end{array}
\right);
\qquad
    \mathcal{L}_2 :
    \left(\begin{array}{cccc}
c_{11}&0\\
0&c_{11}\\
\end{array}
\right);
\qquad
    \mathcal{L}_3 :
    \left(\begin{array}{cccc}
c_{11}&0\\
0&c_{11}\\
\end{array}
\right)
    $. 
\end{itemize}
\end{proposition}

\begin{proposition}
The description of the centroids of every 3-dimensional compatible Ternary Leibniz algebra is given below.
\begin{itemize}
    \item 
    $\mathcal{L}_1 :
    \left(\begin{array}{cccc}
c_{11}&0&0\\
c_{21}&c_{11}&0\\
-c_{21}&0&c_{11}
\end{array}
\right);
\qquad
\mathcal{L}_2 :
    \left(\begin{array}{cccc}
c_{33}&0&0\\
0&c_{33}&0\\
0&0&c_{33}
\end{array}
\right);
\qquad 
\mathcal{L}_3 :
    \left(\begin{array}{cccc}
c_{22}&0&0\\
0&c_{22}&0\\
0&0&c_{22}
\end{array}
\right);
    $
    \item 
    $\mathcal{L}_4 :
    \left(\begin{array}{cccc}
c_{11}&0&0\\
0&c_{11}&0\\
0&0&c_{11}
\end{array}
\right);
\qquad
\mathcal{L}_5 :
    \left(\begin{array}{cccc}
c_{11}&0&0\\
0&c_{11}&0\\
0&0&c_{11}
\end{array}
\right);
\qquad
\mathcal{L}_6 :
    \left(\begin{array}{cccc}
c_{33}&0&0\\
0&c_{33}&0\\
0&0&c_{33}
\end{array}
\right);
    $ 

   $\mathcal{L}_7 :
    \left(\begin{array}{cccc}
c_{11}&0&0\\
0&c_{11}&0\\
0&0&c_{11}
\end{array}
\right);
\qquad
\mathcal{L}_8 :
    \left(\begin{array}{cccc}
c_{11}&0&0\\
0&c_{11}&0\\
0&0&c_{11}
\end{array}
\right);
\qquad\mathcal{L}_9 :
    \left(\begin{array}{cccc}
c_{11}&0&0\\
0&c_{11}&0\\
0&0&c_{11}
\end{array}
\right)
    $.   
\end{itemize}
\end{proposition}
%

 
%
%

\begin{thebibliography}{99}
\bibitem{BW} Basri W., Rakhimov I., Rikhsiboev I. {\it Classification of 3-Dimensional Complex Diassociative
Algebras}, Malaysian Journal of Mathematical Sciences, Vol. 4, No. 2 (2010).

\bibitem{AD} Das A., {\it Cohomology and deformations of compatible Hom-Lie algebras}, Journal of Geometry and Physics, vol 192, october 2023,
 https://doi.org/10.1016/j.geomphys.2023.104951.

\bibitem{BR} Basri W., Rakhimov I., Rikhsiboev I. {\it Four-Dimensional Nilpotent Diassociative Algebras}
Journal of Generalized Lie Theory and Applications, Vol. 9, No. 1 (2015).

\bibitem{IA} Iroda Choriyeva, Abror Khudoyberdiyev, {\it Classification of five-dimensional symmetric Leibniz algebras},
 arXiv:2303.01905v1 [math.RA] 3 Mar 2023.

\bibitem{E} Erik Mainellis, {\it Classification of Low-dimensional Complex Triassociative Algebras}, arXiv:2209.04351v1 [math RA], 9 Sep 2022.


 \bibitem{BD} Dietrich Burde and Willem de Graaf, {\it Classification of Novikov algebras}, 
%

%
%
%

\bibitem{BOK} Beites P.D., Ouaridi A.F., Kaygorodov I., {\it The algebraic and geometric classification of transposed Poisson algebras},
Revista de la Real Academia de Ciencias Exactas, Fsicas y Naturales. Serie A. Matematicas, 117 (2023), no. 2, Paper No. 55, 25 pp.

\bibitem{CSV} Cabrera Casado Yo., Siles Molina M., Velasco M.,  {\it Classification of three-dimensional evolution algebras}, Linear
Algebra and Its Applications, 524 (2017), 68-108.


\bibitem{KSTT} Kobayashi Y., Shirayanagi K., Takahasi S., Tsukada M., {\it A complete classification of three-dimensional algebras
over R and C}, Asian-European Journal of Mathematics, 14 (2021), no. 8, paper no. 2150131, 25 pp.

\bibitem{KST} Kobayashi Yu., Shirayanagi K., Takahasi S.-Ei., Tsukada M., {\it Classification of three-dimensional zeropotent algebras
over an algebraically closed field}, Communications in Algebra, 45 (2017), no. 12, 5037-5052.

\bibitem{P} Petersson H., {\it The classification of two-dimensional nonassociative algebras}, Results in Mathematics, 37 (2000),
no. 12, 120-154.

\bibitem{RRB} Rikhsiboev I. M., Rakhimov I.S., Basri W., {\it Classification of 3-dimensional complex diassociative algebras},
Malaysian Journal of Mathematical Sciences, 4 (2010), no. 2, 241-254.

\bibitem{CK} Canete E.M., Khudoyberdiyev A.Kh., {\it The classification of 4-dimensional Leibniz algebras}, Linear Algebra and its Applications, 
439(1),(2013), 273-288.

\bibitem{CD} Cicalo S., De Graaf W., Schneider C., . {\it Six-dimensional nilpotent Lie algebras, Linear Algebra Appl}, 436(1),
 (2012), 163-189.

\bibitem{KRS} Khudoyberdiyev A.Kh., Rakhimov I.S., Said Husain Sh.K., {\it On classification of 5-dimensional solvable Leibniz algebras},
 Linear Algebra Appl. 457, (2014), 428-454.

\bibitem{SW} Snobl L., Winternitz P., {\it Classification and Identification of Lie algebras}, CRM Monograph series, Volume 33, (2017), 306 pp.

\bibitem{T}   Turkowski P., {\it Solvable Lie algebras of dimension six}, J. Math. Phys. 31 (6), (1990), 1344-1350.

 \bibitem{RM2} Bai R., Meng D., {\it The centroid of n-Lie algebras}, Algebras Groups Geom. 25 (2) (2004) 29-38.

 \bibitem{LG}  Guo L., {\it An introduction to Rota-Baxter algebra}, Internatational Press (US) and Higher Education Press (China), 2012.

 \bibitem{MD} Makhlouf A. and Yau D., {\it Rota-Baxter Hom-Lie admissible algebras}, Communication in Algebra,  $\bf 23$, no 3, 1231-1257, 2014.

\bibitem{GB}  Baxter G., {\it An analytic problem whose solution follows from a simple algebraic identity}, Pacific J. Math.$\bf 10$ (1960),
 731-742.

\bibitem{AM}  Semenov -Tian-Shansky ,  What is classical r-matrix? Funct. Ana. Appl. $\bf 17$ (1983) 259-272.

\bibitem{MA} Makhlouf A. and Amri A., {\it Non-Commutative Ternary Nambu-Poisson Algebras and Ternary Hom-
Nambu-Poisson Algebras},  J Generalized Lie Theory

\bibitem{CA1}  Bai C., {\it A unified algebraic approch to the classical Yang-Baxter equations}, J. Phys. A: Math. Theor. (2007) $\bf 40$, 
11073-11082.

\bibitem{CA2}  Bai C., O. Bellier, L. Guo and X Ni, {\it Splitting of operations, Manin products and Rota-Baxter operators}, Int. Math. Res. Not.
 IMRN, (2012); doi 10.1093/imrnrnr266, arxiv:1106.6080.

\bibitem{GCB} Rota G.-C., {\it Baxter algebras and combinatorial identities, I, II}, Bull. Amer. Math. Soc. A $\bf 75$ (1969), 325-329, 330-334.

\bibitem{CK}  Connes A. and  Kreimer D.,  Renormalisation in quantum field theory and the Riemann-Hilbert problem. I. 
The Hopf algebra strucutre of graphs and the main theorem, Comm. Math. Phys., $\bf 210$ (2000), 249-273.

\bibitem{CJ} Carinena J., Grabowski J. , Marmo G.,  {\it Quantum bi-Hamiltonian systems}, Internat J. Modern Phys A, 2000, 15: 4797-4810.

\bibitem{FN}  Frolicher A., Nijenhuis A.,  {\it Theory of vector valued differential forms}, Part I. Indag Math, 1956, 18: 338-360.

\bibitem{PL} Leroux P., {\it Contruction of Nijenhuis operators and Dendriform trialgebras}, February 2004.

\bibitem{F}  Spitzer F., A combinatorial lemma and its application to probability theory, Trans.Amer. Math. Soc. $bf 82$ (1956), 323-339.

\bibitem{HB} Harris, B., {\it Cohomology of Lie triple systems and Lie algebras withinvolution}, Trans. Amer. Math. Soc. (1961) (98), 148-162.

\bibitem{HT1} Hodge, T. L., {\it Lie triple systems, restricted Lie triple systems and algebraic groups}, J. Algebra, (2001) (244), 533-580.

 \bibitem{TH1}  Hodge, T. L., PARSHALL, B.J., .{\it  On the representarion theory of Lie triple systems}, Trans. Amer. Math. Soc., (2002) (354),
 1, 4359-4391.

\bibitem{IAP}  Kaygorodov I., Pozhidaev A.,  Saraiva P., {\it On ternary generalization of Jordan algebras},

\bibitem{SN}  Okubo S. and  Kamiya N., {\it Jordan-Lie super algebras and Jordan-Lie triple systems}, J. Algebra 198 (1997), 388-411.

\bibitem{IS} Bakayoko I. and Laraiedh I., {\it Ternary Leibniz color algebra and beyond}, Chapter 5 to appear in Proceeding : Algebra without 
border-classical and constructive  Non-associative Algebraic Structures.

%
\end{thebibliography}
\end{document}